\documentclass{amsart}
\usepackage{amssymb,amscd}
\usepackage{graphicx,subfigure,psfrag,color}
\usepackage[all]{xy}

\newtheorem{thm}{Theorem}[section]
\newtheorem{pro}[thm]{Proposition}
\newtheorem{lem}[thm]{Lemma}

\begin{document}
\title{Graph 4-braid groups and Massey products}
\author{Ki Hyoung Ko}
\author{Joon Hyun La}
\author{Hyo Won Park}
\address{Department of Mathematics, Korea Advanced Institute of Science and Technology, Daejeon, 307-701, Korea}
\email{\{knot, mineid1024, h.w.park\}@kaist.ac.kr}
\begin{abstract}
We first show that the braid group over a graph topologically containing no $\Theta$-shape subgraph has a presentation related only by commutators. Then using discrete Morse theory and triple Massey products, we prove that a graph topologically contains none of four prescribed graphs if and only if its 4-braid groups is a right-angled Artin group.
\end{abstract}

\maketitle

\section{Introduction}

A graph is a connected 1-dimensional finite CW-complex in this article. An unordered $n$-tuple $\{\sigma_1,\ldots, \sigma_n\}$ on a graph $\Gamma$ forms an $i$-cube if $i$ of them are edges and the rest are vertices in $\Gamma$. A cube $\{\sigma_1,\ldots, \sigma_n\}$ is {\em off-diagonal} if $\bar\sigma_i\cap\bar\sigma_j=\emptyset$ for all $i \neq j$. All off-diagonal cubes naturally form an $n$-dimensional cube complex $UD_n \Gamma$ called {\em the unordered discrete configuration space} of $\Gamma$.

The unordered topological configuration space of $n$ objects on $\Gamma$ deformation retracts to the unordered discrete configuration space $UD_n \Gamma$ if $\Gamma$ is sufficiently subdivided, more precisely, if each path between two vertices of degree $\neq 2$ passes through at least $n-1$ edges
and each loop at a vertex passes through at least $n+1$ edges \cite{Ab,KKP}.
Under this circumstance, {\em the graph $n$-braid group} $B_n \Gamma$ of $\Gamma$ is the fundamental group $\pi_1 (UD_n \Gamma)$ of the unordered discrete configuration space of $\Gamma$. Since $UD_n \Gamma$ is locally CAT(0) \cite{Ab}, $B_n \Gamma$ is a CAT(0) group.

Besides the properties as a CAT(0) group, the graph braid group $B_n \Gamma$ has other distinctive characteristics depending on the graph $\Gamma$ and the braid index $n$.
In \cite{KP}, the abelianization of $B_n \Gamma$ was completely determined via a
natural decomposition of the underlying graph $\Gamma$. In fact, it is torsion-free for a planar graph and it has 2-torsions for a non-planar graph. If $\Gamma$ is planar then $B_n \Gamma$ has a minimal presentation. That is, $B_n \Gamma$ has a presentation
with $\beta_1$ generators and $B_2 \Gamma$ has a presentation with $\beta_1$ generators and $\beta_2$ relators where $\beta_i$ denotes the $i$-th Betti number of $B_n \Gamma$.
The minimal presentation for $B_n \Gamma$ is automatically {\em commutator-related}, i.e. all relators are words of commutators and the minimal presentation for $B_2 \Gamma$ is automatically {\em simple-commutator-related}, i.e. all relators are commutators.
In the same paper~\cite{KP}, it was shown that if $B_n \Gamma$ is simple-commutator-related for $n\ge 3$, $\Gamma$ does not topologically contain the subgraph of two vertices with four multiple edges between them and the converse was conjectured. In this paper we prove the following weaker version of the conjecture. We say that a graph $\Gamma$ {\em contains} another graph $\Gamma'$ if a subdivision of $\Gamma'$ is a subgraph of a subdivision of $\Gamma$.

\begin{thm}\label{thm:SCRG}
If $\Gamma$ does not contain $N_1$ in Figure~\ref{fig:4nuclei} then $B_n\Gamma$ is simple-commutator-related.
\end{thm}

A right-angled Artin group that has a presentation related by commutators among generators is obviously simple-commutator-related. A question that has been frequently asked since graph braid groups was pioneered by Ghrist and Abrams \cite{AR} is when they are right-angled Artin groups. There have been satisfactory answers for higher braid indices. Farley and Sabalka showed in~\cite{FS2} that a tree $T$ does not contain $N_4$ in Figure~\ref{fig:4nuclei} iff $B_nT$ is right-angled Artin group for $n\ge4$. Kim, Ko and Park  showed in~\cite{KKP} that a graph $\Gamma$ does not contain $N_4$ nor $S$ in Figure~\ref{fig:4nuclei} iff $B_n\Gamma$ is a right-angled Artin group for $n\ge5$.

\begin{figure}[ht]
\centering
\includegraphics[height=2.5cm]{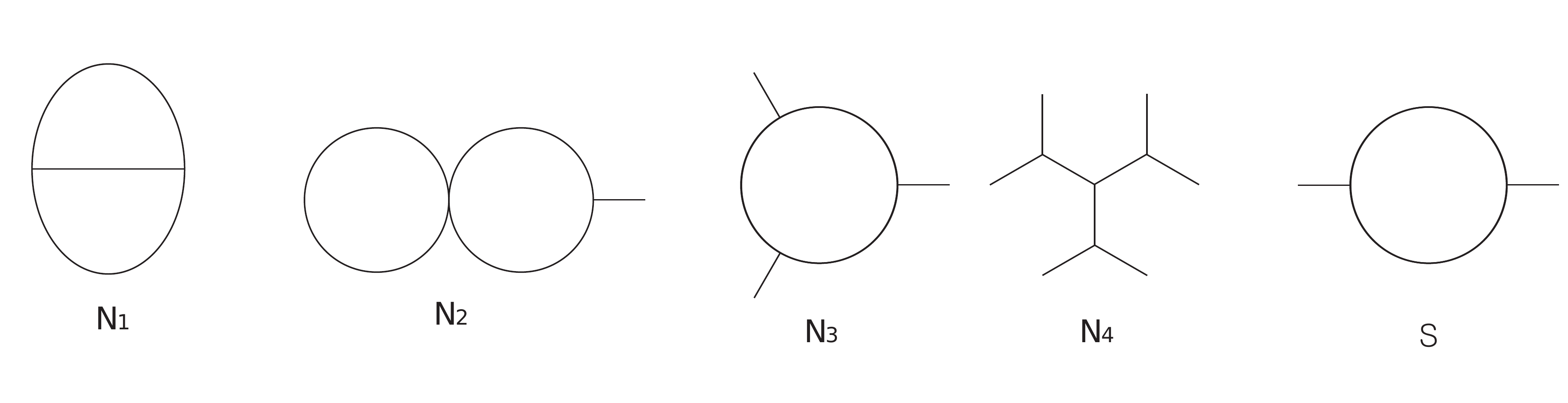}
\caption{$n$-nuclei for $n\ge 4$}
\label{fig:4nuclei}
\end{figure}

A graph $\Gamma$ is an {\em $n$-nuclei} if $B_n \Gamma$ is not a right-angled Artin group and $B_n \Gamma'$ is a right-angled Artin group for every proper subgraph $\Gamma'$ of $\Gamma$ after ignoring vertices of valence 2. It is a reasonable conjecture that a graph $\Gamma$ contains no $n$-nuclei iff $B_n\Gamma$ is a right-angled Artin group. The main result of this article is to verify the conjecture for $n=4$ via the following two theorems

\begin{thm}\label{thm:if}
If $\Gamma$ contains none of $N_1, N_2, N_3, N_4$ in Figure~\ref{fig:4nuclei} then $B_4 \Gamma$ is a right-angled Artin group.
\end{thm}

For the converse, it is enough to prove the following theorem since it was already shown in \cite{KKP} that $B_4\Gamma$ is not a right-angled Artin group for a graph $\Gamma$ containing $N_1$.

\begin{thm}\label{thm:onlyif}
Let $\Gamma$ be a graph that does not contain $N_1$. If $\Gamma$ contains $N_2$, $N_3$ or $N_4$ then $B_4 \Gamma$ is not a right-angled Artin group.
\end{thm}

This article is organized as follows. In \S2, we briefly introduce the triple Massey product and use it to show that $B_4N_i$ are not right-angled Artin groups for $i=2,3,4$. In conjunction with Theorem~\ref{thm:if}, the first four graphs in Figure~\ref{fig:4nuclei} are indeed 4-nuclei. In \S3, we compute a presentation of $B_n\Gamma$ for $\Gamma$ not containing $N_1$ using the discrete Morse theory. In \S4, we prove Theorem~\ref{thm:SCRG} and Theorem~\ref{thm:if} using the computation in \S3. In \S5, we show Theorem~\ref{thm:onlyif} using the triple Massey product.

\section{Triple Massey product and nuclei}\label{s:two}
\subsection{Triple Massey products}\label{ss21:TMP}
We will use the triple Massey product to detect groups that are not right-angled Artin groups. Given a presentation of a group, we need to compute the triple Massey product on the cohomology ring of the group. We begin with recalling relevant results by Fenn and Sjerve \cite{FS4,FS5} and Matei and Suciu \cite{MS,Ma}.

Let $G$ be a group. For cohomology classes $\alpha, \beta, \gamma$ in
$H^1(G)$ such that $\alpha\cup\beta=0$ and
$\beta\cup\gamma=0$, the \emph{triple Massey product} $\langle
\alpha,\beta,\gamma\rangle$ is defined.
We will abuse notations by using the same letter to denote a (co)cycle and its a (co)homology class. Choose 1-cochains $x,y$ such that
$\delta x=\alpha\cup\beta$ and $\delta y=\beta\cup\gamma$ where
$\delta$ denote the coboundary homomorphism. Then $\alpha\cup y+x\cup\gamma$ is
a 2-cocycle and its cohomology class in $H^2(G)$ is well-defined up to the subgroup $\alpha\cup H^1(G)+H^1(G)\cup\gamma$. The triple Massey product $\langle
\alpha,\beta,\gamma\rangle$ is the coset of $\alpha\cup y+x\cup\gamma$ modulo $\alpha\cup H^1(G)+H^1(G)\cup\gamma$ in $H^2(G)$.

Let $G=\langle x_1,\ldots,x_p|r_1,\ldots,r_q\rangle$ be a commutator-related group such that $r_i\in [F,F]$ for all $i$ where $F$ is a free group over $x_1,\ldots,x_p$.
Then we can think of $H_1(G)$ as a free abelian group generated by $x_1\cdots x_p$. Let
$\{x^*_1,\ldots,x^*_p\}$ be the dual basis for
$H^1(G)$. By Hopf isomorphism, $H^2(G)$ is identified
with $K/[F,K]$ where $K$ be the normal subgroup of $F$ generated by
$r_1,\ldots,r_q$. So we think of $H_2(G)$ as a free abelian group generated by
$r_1,\ldots,r_q$. Let $\{r^*_1,\ldots,r^*_q\}$ be the dual bases for $H^2(G)$.

Let $\mathbb ZF$ be a integral group ring of $F$ with the augmentation map
$\varepsilon:\mathbb ZF\to \mathbb Z$ such that
$\varepsilon(x_i)=1$. Recall Fox derivatives $\partial_i:\mathbb
ZF\to \mathbb ZF$ defined by $\partial_i(1)=0$,
$\partial_i(x_j)=\delta_{ij}$ and
$\partial_i(uv)=\partial_i(u)\varepsilon(v)+u\partial_i(v)$. Let $\varepsilon_{i_1\ldots i_k}:\mathbb ZF\to Z$ be
defined by
$\varepsilon_{i_1\ldots i_k}=\varepsilon\partial_{i_1}\cdots\partial_{i_k}$.
Using these notations, cup products and triple Massey products on a cohomology ring of commutator-related group can explicitly be given.

\begin{pro} \label{pro:cup}
Let $G$ be a commutator-related group defined as above.\\
{\rm (1) (Fenn-Sjerve\cite{FS4})}\label{pro:cup}
The cup product $\cup:H^1(G)\wedge H^1(G)\to H^2(G;\mathbb Z)$ is given by
$$x^*_i\cup x^*_j=\sum_{k=1}^p\varepsilon_{ij}(r_k)r^*_k.$$
{\rm (2) (Matei\cite{Ma})}\label{pro:massey}
For cohomology classes $\alpha=\sum a_ix^*_i$, $\beta=\sum b_jx^*_j$, $\gamma=\sum c_kx^*_k$
in $H^1(G)$ such that $\alpha\cup\beta=\beta\cup\gamma=0$,
the Massey product $\langle \alpha,\beta,\gamma\rangle$
is the coset of the cohomology class $\rho$ in $H^2(G)$ modulo $\alpha\cup H^1(G)+H^1(G)\cup\gamma$ where
$$\rho=\sum_{\ell=1}^{q}\sum_{1\le i,j,k\le
p}a_ib_jc_k\varepsilon_{ijk}(r_\ell)r^*_\ell.$$
\end{pro}

The following formulas are immediate from the definition and are useful to apply the proposition.
\begin{align*}
\varepsilon_{k\ell}([u,v])&= \varepsilon_k(u)\varepsilon_\ell(v)-\varepsilon_k(v)\varepsilon_\ell(u)\\
\varepsilon_{k\ell m}([u,v])&= \varepsilon_k(u)\varepsilon_{\ell m}(v)-\varepsilon_m(u)\varepsilon_{k\ell}(v)+
\varepsilon_{k\ell}(u)\varepsilon_m(v)-\varepsilon_k(v)\varepsilon_{\ell m}(u)\\
&+(\varepsilon_k(v)\varepsilon_{\ell}(u)-\varepsilon_k(u)\varepsilon_{\ell}(v))(\varepsilon_m(u)+\varepsilon_m(v)).
\end{align*}
And for $u = x_{i_1} ^{j_1} \cdots x_{i_t} ^{j_t}$,

\begin{align*}
\varepsilon_{k} (u) &= \sum_{s=1} ^{t} j_s \delta_{ki_s}\\
\varepsilon_{k\ell} (u) &= \sum_{1\leq r < s \leq t} j_r j_s \delta_{k i_r} \delta_{\ell i_s} + \sum_{r=1} ^{t} \delta_{k i_r} \delta_{\ell i_r} \frac{j_r - 1}{2} j_r.
\end{align*}

Massey products vanish for cohomology rings of right-angled Artin groups since their Eilenberg-MacLane spaces are formal~\cite{PS}. But we verify this fact via direct computation as a warm-up.

\begin{lem}
Let $G$ be a right-angled Artin group. Every triple Massey product on $H^1(G)$ vanishes.
\end{lem}

\begin{proof}
For each $1\le \ell\le q$, assume $r_\ell=[x_{i_\ell},x_{j_\ell}]$ for some $1\le i_\ell, j_\ell\le p$. Let $\alpha=\sum a_ix^*_i$, $\beta=\sum b_jx^*_j$, $\gamma=\sum c_kx^*_k$ be in $H^1(G)$ such that $\alpha\cup\beta=\beta\cup\gamma=0$. Then by Proposition~\ref{pro:cup}(1), $a_{i_\ell}b_{j_\ell}=a_{j_\ell}b_{i_\ell}$ and $b_{i_\ell}c_{j_\ell}=b_{j_\ell}c_{i_\ell}$ for all $1\le\ell\le q$.
By applying the formulas for $\epsilon_i$ on commutators, the cohomology class $\rho$ in Proposition~\ref{pro:cup}(2) is given by
$$\rho=\sum_{\ell=1}^{q}(a_{j_\ell}b_{i_\ell}-a_{i_\ell}b_{j_\ell})(c_{i_\ell}+c_{j_\ell})=0.$$
\end{proof}

Consequently triple Massey product are obstructions for a commutator-related group to become a right-angled Artin group.

\subsection{Application to simple-commutator-related groups}\label{ss22:Nu}

If $G$ is a simple-commutator-related group, it is easier to compute cup products and triple Massey products on $H^1(G)$. Suppose $X$ and $Y$ are subsets of generators and $R$ is the set of relators for a simple-commutator-related presentation.  If for each $r=[u,v] \in R$,
$$\sum_{x \in X, y \in Y } \varepsilon_x (u) \varepsilon_y (v) - \varepsilon_x (v) \varepsilon_y (u) = 0.$$
then the cup product between $\sum_{x\in X} x^*$ and $\sum_{y\in Y} y^*$ vanishes by Proposition~\ref{pro:cup}(1). In this case we say the pair $(X,Y)$ satisfies the {\em cup zero condition}.

In the proof of the following lemma, we proceed as follows to show a  simple-commutator related group $G$ is not a right-angled Artin group. We first specify three subsets $X$, $Y$ and $Z$ of generators of $G$ such that pairs $(X,Y)$ and $(Y,Z)$ satisfy the cup zero condition so that $\alpha\cup\beta=\beta\cup\gamma=0$ for $\alpha = \sum_{x\in X} x^*$, $\beta = \sum_{y\in Y} y^*$, and $\gamma = \sum_{z\in Z} z^*$ in $H^1 (G)$. We also give a set of generators for the subgroup $H=\alpha\cup H^1 (G)+H^1 (G)\cup\gamma$ of $H^2(G)$ and then give a representative $\rho$ of the triple Massey product $\langle\alpha,\beta,\gamma\rangle$ obtained by Proposition~\ref{pro:cup}(2) that is not contained in the subgroup $H$. Consequently $G$ cannot be a right-angled Artin group.

\begin{lem}\label{lem:3}
Suppose that a group $G$ has one of the following simple-commutator-related presentations with the set $S$ of generators and the set $R$ of relators and for given subsets $X,Y,Z$ of $S$, pairs $(X,Y)$ and $(Y,Z)$ satisfy the cup zero condition.
\begin{enumerate}
  \item $\{a,b,c,x,z\}\subset S$ and $\{[x,a^{-1}b], [x,a^{-1}c],[zaxa^{-1},b^{-1}c],[z,b^{-1}c]\}\subset R$ and let
      $X=\{x\}$, $\{a,b,c\}\subset Y$, $z\not\in Y$ and $Z=\{b\}$.
  \item $\{a,b,x,y\}\subset S$ and
  $\{[x,a^{-1}b], [y,a^{-1}b],[x,baya^{-1}b^{-1}]\}\subset R$ and let $X=Z=\{a,x,y\}$, $\{a,b\}\subset Y$ and $x,y\not\in Y$.
  \item $\{e,x,y\}\subset S$ and $\{[y, \omega_1 x \omega_1^{-1}], [z, \omega_2 x \omega_2^{-1}], [z, e \omega_3 y \omega_3^{-1} e^{-1}]\}\subset R$ and let $X=Z=\{x,y,z\}$, $e\in Y$ and $x,y,z\not\in Y$ where $\omega_i$ are words of generators not in $X\cup Y\cup Z$.
\end{enumerate}
Then $G$ is not a right-angled Artin group.
\end{lem}
\begin{proof}
(1) To use the notations defined above, set $r_1=[x,a^{-1}b]$, $r_2=[x,a^{-1}c]$, $r_3=[zaxa^{-1},b^{-1}c]$, $r_4=[z,b^{-1}c]$.
Then $\alpha\cup a^*=-(r^*_1 + r^*_2+\lambda_1)$, $\alpha\cup b^*=r^*_1 - r^*_3 +\lambda_2$, $\alpha\cup c^*=r^*_2 + r^*_3 +\lambda_3$, $\alpha\cup x^*=\alpha\cup z^*=0$, $a^*\cup\gamma=b^*\cup\gamma=c^*\cup\gamma=0$, $x^*\cup\gamma=r^*_1 - r^*_3 +\lambda_2$ and $z^*\cup\gamma=-(r^*_3 +r^*_4+\lambda_4)$ where $\lambda_j$'s are linear combinations not containing $r^*_1,r^*_2,r^*_3$ and $r^*_4$ in $H^2(G)$.
So $\alpha \cup H^1 (G) = \langle r^*_1 + r^*_2+\lambda_1 , r^*_1 - r^*_3 +\lambda_2, r^*_2 + r^*_3 +\lambda_3,\lbrace \lambda_i \rbrace\rangle$ and $ H^1 (G) \cup \gamma = \langle r^*_1 - r^*_3 +\lambda_2, r^*_3 +r^*_4+\lambda_4,\lbrace \lambda_i \rbrace \rangle$ and so $$H=\langle r^*_1 - r^*_2 +\lambda_1, r^*_1 - r^*_3+\lambda_2, r^*_3 +r^*_4+\lambda_4,\lbrace \lambda_i \rbrace\rangle.$$
\begin{align*}
\rho&= \sum_{r_i\in R} \varepsilon_{\alpha\beta\gamma}(r_i)r^*_i=\sum_{r_i\in R} (\varepsilon_{xab}(r_i)+\varepsilon_{xbb}(r_i)+\varepsilon_{xcb}(r_i))r^*_i\\
&=(\varepsilon_{xab}(r_1)+\varepsilon_{xbb}(r_1))r^*_1+0r^*_2
+(\varepsilon_{xab}(r_3)+\varepsilon_{xbb}(r_3)+\varepsilon_{xcb}(r_3))r^*_3+0r^*_4+\lambda\\
&=(0-1)r^*_1+(1+0+1)r^*_3+\lambda=-r^*_1+2r^*_3+\lambda
\end{align*}
where $\lambda$ is a linear combination not containing $r^*_1,r^*_2,r^*_3$ and $r^*_4$ in $H^2(G)$.
Thus $\rho$ is not in $H$.\\
(2) Set $r_1=[x,a^{-1}b]$, $r_2=[y,a^{-1}b]$, $r_3=[x,baya^{-1}b^{-1}]$.
Then $H=\alpha \cup H^1 (G) = H^1(G) \cup \gamma = \langle r^*_1 + r^*_2+\lambda_1, r^*_1 - r^*_3 +\lambda_2, r^*_2 + r^*_3+\lambda_3,\lbrace \lambda_i \rbrace  \rangle$. Thus $\rho = r^*_1 + r^*_2 + 4r^*_3+ \lambda$ is not in $H$ where $\lambda_j$'s and $\lambda$ are linear combinations not containing $r^*_1,r^*_2$ and $r^*_3$ in $H^2(G)$.\\
(3) Set $r_1=[y, \omega_1 x \omega_1^{-1}]$, $r_2=[z, \omega_2 x \omega_2^{-1}]$, $r_3=[z, e \omega_3 y \omega_3^{-1} e^{-1}]$. Then $H=\alpha \cup H^1 (G) = H^1(G) \cup \gamma = \langle r^*_1 + r^*_2 + \lambda_1 , r^*_1 - r^*_3 + \lambda_2 , r^*_2 + r^*_3 + \lambda_3 , \lbrace \lambda_i \rbrace \rangle$ and so $\rho = 2r^*_3 + \lambda$ is not in $H$ where $\lambda_j$'s and $\lambda$ are linear combinations not containing $r^*_1,r^*_2$ and $r^*_3$ in $H^2(G)$.
\end{proof}

\begin{lem}\label{lem:4}
$B_4 N_2$, $B_4 N_3$ and $B_4 N_4$ are not a right-angled Artin groups.
\end{lem}
\begin{proof}
Simple-commutator-related presentations for $B_4 N_2$, $B_4 N_3$ and $B_4 N_4$ can be easily derived by using the result of Section~\ref{ss41:scrg} as follows and they satisfies the hypotheses of Lemma~\ref{lem:3} (See Lemma~\ref{lem:oif} for details).\\
(1) $B_4 N_2=\langle x,y,x,a,b,c\:|\: r_1, r_2, r_3, r_4, r_5 \rangle *F_5$\\
where $r_1=[x,a^{-1}b]$, $r_2=[x,a^{-1}c]$, $r_3=[zaxa^{-1}, b^{-1}c]$, $r_4=[y,b^{-1}c]$, and $r_5=[z,b^{-1}c]$. And $X=\{ x\}$, $Y=\{a,b,c\}$ and $Z=\{b\}$\\
(2) $B_4 N_3=\langle x,y,a,b \:|\: r_1, r_2, r_3 \rangle *F_{6}$\\
where $r_1=[x,a^{-1}b]$, $r_2=[y,a^{-1}b]$, and $r_3=[x,baya^{-1}b^{-1}]$. And $X=Z\{a, x, y\}$, $Y=\{a,b\}$\\
(3) $B_4 N_4=\langle x,y,z,a,b,c,d,e\:|\: r_1,r_2,r_3,r_4,r_5,r_6 \rangle *F_{16}$\\
where $r_1=[x, y]$, $r_2=[x, z]$, $r_3=[z, ebyb^{-1}e^{-1}]$, $r_4=[x,a]$, $r_5=[y, b^{-1}c]$, and $r_6=[z,d]$. And $X=Z=\{x,y,z\}$, $Y=\{e\}$
\end{proof}

\section{Discrete Morse theory and cactus graphs}\label{s:three}

\subsection{Discrete Morse theory}\label{ss31:DM}
To compute presentations of graph braid groups, we use the algorithm of Farley and Sabalka in \cite{FS1} based on the discrete Morse theory developed by Forman in~\cite{For}. Following \cite{KKP}, we review the algorithm briefly.

First we choose a maximal tree $T$ in $\Gamma$. We assume that $T$ is sufficiently subdivided and embedded in a plane. Edges in $\Gamma - T$ are called {\em deleted edges}. As the base vertex, pick a vertex of degree $1$ in $T$ or a vertex of degree 2 if there are no vertices of degree 1. We require that a path between the base vertex and any vertex of degree $\ge 3$ is also sufficiently subdivided.  Then take a regular neighborhood $N$ of $T$. Now the boundary $\partial N$ is a simple closed curve. Now starting from the base vertex numbered 0, we number unvisited vertices as traveling along $\partial N$ clockwise. Each edge $e$ in $\Gamma$ is oriented so that the initial vertex $\iota(e)$ is larger than the terminal vertex $\tau(e)$.
For a cell $c = \{c_1, \ldots, c_{n-1}, v\} \in UD_n \Gamma$, a vertex $v$($\ne 0$) in $c$ is \textit{blocked} if for the edge $e$ in $T$ with $\iota(e) = v$, $\tau(e)$ is either in $c$ or an end of an edge in $c$. We assume the base vertex $0$ is always blocked. An edge $e$ in $c$ is {\em order-respecting} if $e$ is not a deleted edge and there is no vertex $v$ in $c$ such that $v$ is adjacent to $\tau(e)$ in $T$ and $\tau(e) < v < \iota(e)$.

A cell is {\em critical} if it contains neither unblocked vertices nor order-respecting edges. The unordered discrete configuration space $UD_n\Gamma$ collapses to the CW complex consisted of critical cells.
There are one critical 0-cell and $B_n\Gamma$ has a presentation generated by critical 1-cells and related by boundaries of critical 2-cells. However the boundary of a critical 2-cell may not be a word of critical 1-cells without rewriting that corresponds to collapsing. Let $c$ be an 1-cell in a word $w$ and let $e$ denote the edge in $c$. If $c$ is not critical, there are two possibilities.
If $e$ is order-respecting and there is no unblocked vertex $v$ in $c$ such that $v<\tau(e)$, we erase $c$. Otherwise let $v$ be the smallest unblocked vertex in $c$ and $e'$ be an edge in $T$ such that $\iota(e')=v$. We rewrite $c$ by the product $c_1c_2c_3^{-1}$ of three 1-cells where $c_1$, $c_2$ and $c_3$ are obtained from $c$ by replacing the pair $(e, v)$ by $(\iota(e), e')$, $(e, \tau(e'))$ and $(\tau(e), e')$. We iterate rewriting $w$ until it becomes a word $\tilde r(w)$ on critical 1-cells. The following lemma often makes computation shorter.

\begin{lem}\label{lem:red1}{\rm (\cite{KP})}
Let $c$ be a 1-cell in $UD_n \Gamma$ and $e$ be an edge in $\Gamma$ such that $\iota(e)$ is an unblocked vertex in $c$. If there is no vertex $w$ that is either in $c$ or an end vertex of an edge in $c$ and satisfies $\tau(e)< w<\iota(e)$. Then $\tilde r(c)=\tilde r(V_e(c))$ where $V_e(c)$ denotes the $1$-cell obtained from $c$ by replacing $\iota(e)$ by $\tau(e)$.
\end{lem}

Resulting presentations heavily depends on choices of a maximal tree and an order on vertices. We will make a choice of a maximal tree that serve our purpose. From now on we only consider graphs that do not contain $N_1$. Such graphs are called {\em cactus} graphs and have an {\em outer-planar} embedding, i.e., the unbounded face contains all vertices.

Given a cactus graph $\Gamma$ on a plane, choose a vertex of degree one as the base vertex. If there is no vertices of degree one, choose a terminal cycle, i.e. a cycle that shares only one vertex with other cycles and choose a vertex of degree 2 on the terminal cycle. The other vertex of the edge is chosen as the base vertex. As we travel along $\Gamma$ starting at the base vertex, we always go to the leftmost edge at a fork vertex and turn back at a vertex of degree 1. If we enter an edge of a cycle, then delete the edge. Continue this procedure until we have a maximal tree $T$ of $\Gamma$. We number the vertices as explained above. Note that for every deleted edge $d$, $\tau(d)$ is the smallest and $\iota(d)$ is the largest vertex in the cycle containing $d$.

We now define few notations. For each vertex $v$, there is a unique edge path $\gamma_v$ from $v$ to the base vertex 0 in $T$. For vertices $v$, $w$ in $\Gamma$, $v\wedge w$ denotes the largest vertex of degree $\ge 3$ in $\gamma_v \cap\gamma_w$. Obviously, $v\wedge w\le v$ and $v\wedge w\le w$. A branch of $v$ is a component of $T - \lbrace v \rbrace$ and the branches of $v$ are sequentially numbered clockwise starting the number $0$ on the branch containing $\gamma_v$. Let $\mu(v)$ be the largest branch number. Then $\mu(v)$ is one less that the degree of $v$ in $T$.
For a pair of vertices $(v, w)$ with $v < w$, $g(v,w)$ will denote the number for the branch of $v$ containing $w$. If $v=w\wedge v$, $g(v,w)\ge 1$ and if $v>w\wedge v$, $g(v,w)=0$. If we choose a maximal tree and the numbering on vertices as above, it is easy to check that the following properties hold.

\begin{itemize}
\item[(T1)] The initial (terminal, respectively) vertices of all deleted edges are vertices of degree 2 ($\ge3$) in $\Gamma$. Note that the base vertex can also be a terminal vertex of a deleted edge if it is of degree 2;
\item[(T2)] There is no pair of deleted edges $d$ and $d'$ such that $\tau(d)=\tau(d')$ and $g(\tau(d),\iota(d))=g(\tau(d'),\iota(d'))$;
\item[(T3)] Let $d$ be a deleted edge. If $v$ is a vertex lying on the path  between $\tau(d)$ and $\iota(d)$ in $T$ then $g(v,\iota(d))=1$.
\item[(T4)] If $v_1 \wedge v_2 < v_1$ and $v_1 \wedge v_2 < v_2$, then every path in $\Gamma$ between $v_1$ and $v_2$ must pass $v_1 \wedge v_2$.
\end{itemize}

We recall the notation from \cite{FS1, KKP} that is convenient to represent 1- or 2-cells in $UD_n\Gamma$. Let $A$ be a vertex of degree $\mu+1$ in a maximal tree $T$ of $\Gamma$. Let $\vec a$ be a vector $(a_1,\ldots,a_\mu)$ of nonnegative integers and let $|\vec a|=\sum_{i=1}^\mu a_i$. And $\vec\delta_k$ denotes the $k$-th coordinate unit vector. %Here we adopt the convention that $\vec\delta_0$ is the zero vector.
Then for $1\le k\le \mu$, $A_k(\vec a)$ denotes the set consisted of one edge $e$ with $\tau(e)=A$ that lies on the $k$-th branch of $A$ together with $a_i$ blocked vertices that lie on the $i$-th branch of $A$. The edge $e$ will be denoted by $A_k$. Furthermore a deleted edge $d$ with $\tau(d)=A\ne0$ and $g(A,\iota(d))=k$ is denoted by $A_{-k}$ and this is well-defined by the property (T2). Note that there are no vertices blocked by the initial vertex of any deleted edge by (T1). Let $A_{-k}(\vec a)$ denote the set consisting of the deleted edge $d$ together with $a_i$ blocked vertices that lie on the $i$-th branch of $A=\tau(d)$ for each $i$. The deleted edge $d$ with $\tau(d)=0$ and $\iota(d)=A$ is denoted by $A_0$. This definition is slightly different from that used in \cite{FS1, KKP} and is more convenient for this work. For $1\le s\le n$, $0_s$ denotes the set $\{0,1,\ldots,s-1\}$ of $s$ consecutive vertices from the base vertex. Let $\dot A(\vec a)$ denote the set of vertices consisting of $A$ together with $a_i$ blocked vertices that lies on the $i$-th branch and let $A(\vec a)=\dot A(\vec a)-\{A\}$.
Every critical 1- and 2-cell can be represented by the following forms:
$$A_k(\vec a)\cup 0_s,\quad A_k(\vec a)\cup B_\ell(\vec b)\cup 0_s$$
where $A$ and $B$ are vertices of degree $\ge 3$ in $\Gamma$. Furthermore, since $s$ is uniquely determined by $s=n-1-|\vec a|$ or $n-(2+|\vec a|+|\vec b|)$, no confusion will occur even if we omit $0_s$ in the notation. Let $\vec a-1$ denote the vector obtained from $\vec a$ by subtracting 1 from the first positive entry. Then $\vec a-\alpha$ is defined recursively by $\vec a-(\alpha+1) = (\vec a-\alpha)-1$. Also we denote the index of the first nonzero entry of $\vec{a}$ by $p(\vec a)$, that is, $p(\vec a)=i$ iff $a_i$ is the first nonzero entry of $\vec a$. For $1\le k\le \mu$, let $(\vec a)_k=(a_1,\ldots,a_{k},0,\ldots,0)$ and $|\vec a|_k=a_1+\cdots+a_{k}$.

\subsection{Computation for cactus graph braid groups}\label{ss32:P1}

Given a cactus graph $\Gamma$, we always assume that we use a maximal tree $T$ and the numbering on vertices as given in \S\ref{ss31:DM}. Suppose that $A_k(\vec a)\cup B_\ell(\vec b)$ is a critical 2-cell in $UD_n \Gamma$ such that $A<B$. Here $k$ and $\ell$ are allowed to be negative integers to accommodate deleted edges. The path from $\iota(A_k)$ to the base vertex contains $A$ even if $A_k$ is a deleted edge and so we always have $A \leq B \wedge \iota(A_k)$. Thus there are four possibilities:
\begin{enumerate}
\item $A \wedge B < A$
\item $A \wedge B = A < B \wedge \iota(A_k) < B$
\item $A \wedge B = A = B \wedge \iota(A_k)$
\item $A \wedge B = A < B \wedge \iota(A_k) = B$
\end{enumerate}
In the rest of the article, we will frequently make analyses via these four cases. Note that the case (2) and (4) occur only when $k<0$, i.e. $A_k$ is a deleted edge. Both $A$ and $B\wedge\iota(A_k)$ lies on a cycle in the case (2) and both $A$ and $B$ lie on a cycle in the case (4).

Images under rewriting $\tilde r$ tend to be long and complicated and so we adapt
the following two notations.
For a vertex $A$ of degree $\ge3$, a vector $\vec a$ defined at $A$, and integers $1\le\ell\le\mu(A)$, 1$\le m\le n-|\vec a|$, let
$$\mathbf{A}(\vec a,\ell,m)=\tilde r(\prod^{|\vec a|-1}_{\alpha = 0} A(\alpha))$$
where $A(\alpha)=A_{p(\vec a-\alpha)}((\vec a -\alpha-1)+m\vec\delta_\ell)$. We observe a few immediate properties. If $\vec a=\vec 0$, then $\mathbf{A}(\vec a,\ell,m)=1$. If $A(\alpha)$ is not critical 1-cell, then $\ell\ge p(\vec a-\alpha)$ and it is collapsable and so $A(\alpha)=1$.  And if $A_k(\vec b)$ is a critical 1-cell in $\mathbf{A}(\vec a,1,1)$, then $(\vec b)_{k-1}=\vec\delta_1$.

For vertices $A,B$ of degree $\ge3$ such that $A<B$ and vectors $\vec a,\vec b$ defined at $A,B$, let
$$\mathbf{(B,A)}(\vec b,\vec a)=\tilde r(\prod^{|\vec b|-1}_{\alpha = 0}\mathbf{A}(\alpha)\cdot B(\alpha)\cdot(\mathbf{A}(\alpha))^{-1})$$
where $\mathbf{A}(\alpha)=\mathbf{A}(\vec a,g(A,B),|\vec b|+1-\alpha)$ and $B(\alpha)=B_{p(\vec b-\alpha)}((\vec b -\alpha-1)+\vec\delta_1)$. If $B(\alpha)$ is not a critical 1-cell, then $\tilde r(B(\alpha))=1$.

The following are less obvious.
\begin{pro}
    \begin{enumerate}
        \item If $\vec a=(\vec a)_{\ell}$ then $\mathbf{A}(\vec a,\ell,m)=1$
        \item If $\vec a=(\vec a)_{g(A,B)}$ then $\mathbf{(B,A)}(\vec b,\vec a)=\mathbf{B}(\vec b,1,1)$
        \item If $\vec a-(\vec a)_{\ell}=\vec v-(\vec v)_{\ell}$ then $\mathbf{A}(\vec a,\ell,m)=\mathbf{A}(\vec v,\ell,m)$
        \item If $\vec a-(\vec a)_{g(A,B)}=\vec v-(\vec v)_{g(A,B)}$ then $\mathbf{(B,A)}(\vec b,\vec a)=\mathbf{(B,A)}(\vec b,\vec v)$
      %  \item
    \end{enumerate}
\end{pro}
\begin{proof}
(1) $\vec a=(\vec a)_\ell$ implies that $p(\vec a-\alpha)\le \ell$ and so $A(\alpha)$ for all $\alpha$ is not a critical 1-cell. Thus $\tilde r(A(\alpha))=1$ and so $\mathbf{A}(\vec a,\ell,m)=1$

(2) Since $\mathbf{A}(\alpha)=1$ for all $\alpha$,
$$\mathbf{(B,A)}(\vec b,\vec a)=\tilde r(\prod^{|\vec b|-1}_{\alpha = 0}\mathbf{A}(\alpha) \cdot B(\alpha)\cdot (\mathbf{A}(\alpha))^{-1})=\tilde r(\prod^{|\vec b|-1}_{\alpha = 0} B(\alpha))=\mathbf{B}(\vec b,1,1).$$

(3) If $p(\vec a-\alpha)\le \ell$ then $A(\alpha)$ is not a critical 1-cell. So
\begin{align*}
\mathbf{A}(\vec a,\ell,m)&=\tilde r(\prod^{|\vec a|-1}_{\alpha = 0} A(\alpha))=\tilde r(\prod^{|\vec a|-1}_{\alpha = |\vec a|_\ell-1} A(\alpha))\\
&=\tilde r(\prod^{|\vec v|-1}_{\alpha = |\vec v|_\ell-1} A(\alpha))=\tilde r(\prod^{|\vec v|-1}_{\alpha = 0} A(\alpha))=\mathbf{A}(\vec v,\ell,m)
\end{align*}

(4) Since $\mathbf{A}(\alpha)=\mathbf{A}(\vec a,\ell,|\vec b|+1-\alpha)=\mathbf{A}(\vec v,\ell,|\vec b|+1-\alpha)$ for all $\alpha$ where $\ell=g(A,B)$. So
$$\mathbf{(B,A)}(\vec b,\vec a)=\tilde r(\prod^{|\vec b|-1}_{\alpha = 0}\mathbf{A}(\alpha) \cdot B(\alpha)\cdot (\mathbf{A}(\alpha))^{-1})=\mathbf{(B,A)}(\vec b,\vec v).$$
\end{proof}

\begin{lem}\label{rel1}
Let $\Gamma$ be a cactus graph. Let $c$ be a critical 2-cell in $UD_n\Gamma$ of the form $A_k(\vec a)\cup B_\ell(\vec b)$ with $A<B$. Then
    \begin{enumerate}
        \item If $A \wedge B < A$,
            $$\tilde r(\partial c)=[B_\ell(\vec b),\omega \cdot A_k(\vec a)\cdot \omega^{-1}]$$
            where $X=A \wedge B$ and $\omega= \mathbf{X}((|\vec b|+1)\vec\delta_{g(X,B)},g(X,A),|\vec a|+1)$.
        \item If $A \wedge B = A < B \wedge \iota(A_k) < B$,
            $$\tilde r(\partial c)=[B_\ell(\vec b),\omega^{-1}\cdot \gamma\cdot A_k(\vec a+(|\vec b|+1)\vec\delta_{|k|})\cdot\omega]$$
            where $\omega=\mathbf{A}(\vec a,|k|,|\vec b|+1)$ and $\gamma=\mathbf{C}((|\vec b|+1)\vec\delta_{g(C,B)},1,1)$ for $C=B \wedge \iota(A_k)$.
        \item If $A \wedge B = A = B \wedge \iota(A_k)$,
            $$\tilde r(\partial c)=[B_\ell(\vec b),\omega_1^{-1}\cdot A_k(\vec a+(|\vec b|+1)\vec\delta_{g(A,B)})\cdot\omega_2]$$
            where $\omega_1=\mathbf{A}(\vec a+\vec\delta_{|k|},g(A,B),|\vec b|+1)$ and $\omega_2=\mathbf{A}(\vec a,g(A,B),|\vec b|+1)$.
        \item If $A \wedge B = A < B \wedge \iota(A_k) = B$, then $\tilde r(\partial c)$ gives the relation
            \begin{align*}
            &\:\omega_1\cdot B_\ell(\vec b+\vec\delta_1)\cdot\omega_1^{-1}\\
            =&\:\beta_1\cdot A_k(\vec a+(|\vec b|+1)\vec\delta_{|k|})\cdot\omega_2 \cdot B_\ell(\vec b)\cdot\omega_2^{-1}\cdot A_k(\vec a+(|\vec b|+1)\vec\delta_{|k|})^{-1}\cdot\beta_2^{-1}
            \end{align*}
            where $\omega_1=\mathbf{A}(\vec a,|k|,|\vec b|+2)$,  $\omega_2=\mathbf{A}(\vec a,|k|,|\vec b|+1)$, $\beta_2=(\mathbf{B},\mathbf{A})(\vec b,\vec a)$ and
            $\beta_1=(\mathbf{B},\mathbf{A})(\vec b+\vec\delta_{|\ell|},\vec a)$.
    \end{enumerate}
\end{lem}
\begin{proof}
The boundary $\partial(c)$ of the 2-cell $c=A_k(\vec a)\cup B_\ell(\vec b)$ is given by
$$A_k(\vec a)\cup B(\vec b)\cup\{\iota(B_\ell)\}\cdot \dot A(\vec a)\cup B_\ell(\vec b)\cdot(A_k(\vec a)\cup \dot B(\vec b))^{-1}\cdot(A(\vec a)\cup B_\ell(\vec b)\cup\{\iota(A_k)\})^{-1}.$$

\noindent (1) By applying Lemma~\ref{lem:red1} repeatedly to the smallest unblocked vertices, the first and the third terms under $\tilde r$ are
$$\tilde r(A_k(\vec a)\cup B(\vec b)\cup\{\iota(B_\ell)\})=\tilde r(A_k(\vec a)\cup\dot B(\vec b))=\tilde r(A_k(\vec a)\cup X((|\vec b|+1)\vec\delta_\nu))$$
where $\nu=g(X,B)$. Let $\mu=g(X,A)$. %??? \mu is not used.
By the definition of $\tilde r$ and the induction on $|\vec b|$, we have
\begin{align*}
\tilde r(A_k(\vec a)\cup X((|\vec b|+1)\vec\delta_\nu))=&\ \tilde r(A(\vec a)\cup X_\nu(|\vec b|\vec\delta_\nu)\cup\{\iota(A_k)\})\cdot\tilde r(A_k(\vec a)\cup\dot X(|\vec b|\vec\delta_\nu))\\
&\cdot \tilde r\{(\dot A(\vec a)\cup X_\nu(|\vec b|\vec\delta_\nu))^{-1}\} \\
=&\ (X_\nu(|\vec b|\vec\delta_\nu+(|\vec a|+1)\vec\delta_1))\cdot\tilde r(A_k(\vec a)\cup X(|\vec b|\vec\delta_\nu))\\
&\cdot(X_\nu(|\vec b|\vec\delta_\nu+(|\vec a|+1)\vec\delta_1))^{-1} \\
=&\ \omega\cdot A_k(\vec a)\cdot \omega^{-1}
\end{align*}

Similarly we can rewrite the second and the fourth terms as follows:
\begin{align*}
\tilde{r}(\dot A(\vec{a}) \cup B_\ell(\vec{b})) = \tilde{r}( A(\vec{a}) \cup B_\ell(\vec{b})\cup \{\iota(A_k) \}) = B_\ell(\vec b).
\end{align*}

\noindent(2) $A < B\wedge\iota(A_k)$ implies that $A_k$ is a deleted edge and so $k\le0$. Moreover $g(A,B)=|k|$ and $g(C,\iota(A_k))=1$. By applying Lemma~\ref{lem:red1} repeatedly to the smallest unblocked vertices, we have
$$\tilde{r}(A(\vec{a}) \cup B_\ell(\vec{b}) \cup \{\iota(A_k)\}) = \tilde{r}(\dot A(\vec{a}) \cup B_\ell(\vec{b}) ) = \tilde{r} (B_\ell(\vec{b}) \cup A(\vec a-(\vec a)_{|k|})).$$
Moreover,
\begin{align*}
\tilde{r} (B_\ell(\vec{b}) \cup A(\vec v)) &=\mathbf{A}(\vec v,|k|,|\vec b|+1)\cdot B_\ell(\vec{b}) \cdot(\mathbf{A}(\vec v,|k|,|\vec b|+1))^{-1}\\
&=\omega \cdot B_\ell(\vec{b}) \cdot \omega^{-1}
\end{align*}
where $\vec v = \vec a-(\vec a)_{|k|}$,
since $\mathbf{A}(\vec v,|k|,|\vec b|+1)=\mathbf{A}(\vec a,|k|,|\vec b|+1)$.

Let $\nu=g(C,B)$. By Lemma~\ref{lem:red1} we see
$$\tilde r(A_k(\vec a)\cup B(\vec b)\cup\{\iota(B_\ell)\})=\tilde r(A_k(\vec a)\cup\dot B(\vec b))=\tilde r(A_k(\vec a)\cup C((|\vec b|+1)\vec\delta_\nu)).$$
By the definition of $\tilde r$ and the induction on $|\vec b|$, we have
\begin{align*}
\tilde r(A_k(\vec a)\cup C((|\vec b|+1)\vec\delta_\nu))=&\ \tilde r(A(\vec a)\cup C_\nu(|\vec b|\vec\delta_\nu)\cup\{\iota(A_k)\})\cdot\tilde r(A_k(\vec a)\cup\dot C(|\vec b|\vec\delta_\nu))\\
&\cdot \tilde r\{(\dot A(\vec a)\cup C_\nu(|\vec b|\vec\delta_\nu))^{-1}\} \\
=&(C_\nu(|\vec b|\vec\delta_\nu+\vec\delta_1))\cdot\tilde r(A_k(\vec a+\vec\delta_{|k|})\cup C(|\vec b|\vec\delta_\nu))\\
=&\mathbf{C}((|\vec b|+1)\vec\delta_\nu,1)\cdot A_k(\vec a+(|\vec b|+1)\vec\delta_{|k|})
\end{align*}

\noindent(3) Let $\nu=g(A,B)$. Lemma~\ref{lem:red1} implies
$$\tilde r(A_k(\vec a)\cup B(\vec b)\cup\{\iota(B_\ell)\})=\tilde r(A_k(\vec a)\cup\dot B(\vec b))=A_k(\vec a+(|\vec b|+1)\vec\delta_\nu).$$

For the other two terms, we compute as in the case of (2) and obtain
\begin{align*}
 \tilde{r}(\dot A(\vec{a}) \cup B_\ell(\vec{b}) ) &= \tilde{r} (B_\ell(\vec{b}) \cup A(\vec a-(\vec a)_\nu))= \omega_2\cdot B_\ell(\vec b) \cdot \omega_2^{-1}\\
 \tilde{r}(A(\vec{a}) \cup B_\ell(\vec{b})\cup\{\iota(A_k) \}) &= \tilde{r} (B_\ell(\vec{b}) \cup A(\vec a+\vec\delta_{|k|}-(\vec a+\vec\delta_{|k|})_\nu))= \omega_1 \cdot B_\ell(\vec b) \cdot \omega_1^{-1}
\end{align*}

\noindent(4) The hypothesis implies $k\le 0$ and $g(A, B) = |k|$. Then
\begin{align*}
\tilde{r} (\dot A(\vec{a}) \cup B_\ell(\vec{b}) ) &=\mathbf{A}(\vec a,|k|,|\vec b|+1)\cdot B_\ell(\vec b)\cdot (\mathbf{A}(\vec a, |k|,|\vec b|+1))^{-1}\\
\tilde{r} ( A(\vec{a}) \cup B_\ell(\vec{b})\cup\{\iota(A_k) \} ) &=\tilde r (A(\vec a)\cup B_\ell(\vec b+\vec\delta_1))\\
&=\mathbf{A}(\vec a,|k|,|\vec b|+2)\cdot B_\ell(\vec b+\vec\delta_1)\cdot(\mathbf{A}(\vec a,|k|,|\vec b|+2))^{-1}.
\end{align*}

By induction on $|\vec b|$ and $\tilde r(A(\vec a)\cup B_\ell(\vec b))$, we have
\begin{align*}
\tilde r (A_k(\vec a)\cup B(\vec b))&=\tilde r (A(\vec a)\cup B_{p(\vec b)}((\vec b-1)+\vec\delta_1))\cdot\tilde r(A_k(\vec a+\vec\delta_{|k|})\cup B(\vec b-1))\\
&= \mathbf{(B,A)}(\vec b,\vec a)\cdot A_k(\vec a+|\vec b|\vec\delta_{|k|}))
\end{align*}

And so
\begin{align*}
\tilde{r} (A_k(\vec{a}) \cup B(\vec{b})\cup\{\iota(B_\ell) \})&=\tilde r (A_k(\vec a)\cup B(\vec b+\vec\delta_{|\ell|}))\\
&=\mathbf{(B,A)}(\vec b+\vec\delta_{|\ell|},\vec a)\cdot A_k(\vec a+(|\vec b|+1)\vec\delta_{|k|}),\\
\tilde{r} (A_k(\vec{a}) \cup \dot B(\vec{b}))&=\tilde r (A_k(\vec a+\vec\delta_{|k|})\cup B(\vec b))\\
&=\mathbf{(B,A)}(\vec b,\vec a)\cdot  A_k(\vec a+(|\vec b|+1)\vec\delta_{|k|}).
\end{align*}

Therefore $\tilde{r} (\partial c)$ is no longer a commutator:
\begin{align*}
\tilde{r} (\partial c) =  &\omega_1 \cdot B_\ell(\vec b+\vec\delta_1)\cdot \omega_1^{-1} \cdot\beta_2\cdot A_k(\vec a+(|\vec b|+1)\vec\delta_{|k|})\cdot\omega_2\cdot B_\ell(\vec b)^{-1}\omega_2^{-1}\\
&\cdot A_k(\vec a+(|\vec b|+1)\vec\delta_{|k|})^{-1}\cdot\beta_1^{-1}
\end{align*}
which gives the relation
\begin{align*}
\omega_1\cdot& B_\ell(\vec b+\vec\delta_1)\cdot\omega_1^{-1}\\
=&\beta_1\cdot A_k(\vec a+(|\vec b|+1)\vec\delta_{|k|})\cdot\omega_2 \cdot B_\ell(\vec b)\cdot\omega_2^{-1}\cdot A_k(\vec a+(|\vec b|+1)\vec\delta_{|k|})^{-1}\cdot\beta_2^{-1}
\end{align*}

\end{proof}

The following lemma is useful for simplifying relators in the next chapter.
\begin{lem}\label{lem:AB}
Let $A$ and $B$ be vertices with $A<B$ and satisfy $A\wedge B=A<B$. Then we have
\begin{align*}
&{\mathbf{A} (\vec{a}, k, |\vec b|+2)} = (\mathbf{(B, A)} (\vec{b}, \vec{a}))^{-1}\cdot \mathbf{A} (\vec{a}, k, |\vec b|+2)\cdot \mathbf{B} (\vec{b}, 1,1) .%\\
%&= {\mathbf{B} (\vec{b} + \vec{\delta_{|\ell|}} - 1, 1,1)}^{-1} {\mathbf{A} (\vec{a}, |k|, |\vec b|+2)}^{-1} \mathbf{(B, A)} (\vec{b} + \vec{\delta_{|\ell|}} -1, \vec{a}, |k|)
\end{align*}
where $\vec a$ and $\vec b$ are vectors defined at $A$ and $B$ respectively and $k=g(A,B)$.
\end{lem}

\begin{proof}
We use the induction on $(|\vec a|,|\vec b|)$ with the lexicographical order. Suppose $|\vec a|=0$. Then $\mathbf{A} (\vec{a}, k, |\vec b|+2)=1$ and $\mathbf{(B, A)} (\vec{b}, \vec{a})=\mathbf{B} (\vec{b}, 1,1)$. Thus the formula holds for any $|\vec b|$.

Let $\mathbf{(B, A)} (|\vec a|,|\vec b|)=\mathbf{(B, A)} (\vec{b}, \vec{a})$, $\mathbf{A}(|\vec a|,|\vec b|)=\mathbf{A} (\vec{a}, k, |\vec b|+2)$, and $A(0)=A_{p(\vec a)}((\vec a-1)+(|\vec b|+2)\vec\delta_{k})$.
Now suppose that $|\vec a|=m+1$ and $|\vec b|=n+1$.
Since $\mathbf{A}(m+1,n+1)=\tilde r(A(0))\cdot\mathbf{A}(m,n+1)$, we have
\begin{align*}
\mathbf{(B, A)}& (m+1,n+1)^{-1}\cdot \mathbf{A} (m+1,n+1)\cdot \mathbf{B} (\vec{b}, 1,1) \\
=&\mathbf{(B, A)} (m+1,n+1)^{-1}\cdot \tilde r(A(0))\cdot\mathbf{(B, A)} (m,n+1)\\
&\cdot\mathbf{(B, A)} (m,n+1)^{-1}\cdot \mathbf{A} (m,n+1)\cdot \mathbf{B} (\vec{b}, 1,1)\\
=&\mathbf{(B, A)} (m+1,n+1)^{-1}\cdot \tilde r(A(0))\cdot\mathbf{(B, A)} (m,n+1) \cdot \mathbf{A}(m,n+1)
%=&\mathbf{B} (\vec{b}, 1,1)^{-1}{\mathbf{A} (\vec{a}-1, k, |\vec b|+2)}^{-1}\mathbf{(B, A)} (\vec{b}, \vec{a}-1, k)\\
%&\mathbf{(B, A)} (\vec{b}, \vec{a}-1, k)^{-1}r(A_{p(\vec a)}((\vec a-1)+(|\vec b|+2)\vec\delta_{k}))^{-1}\mathbf{(B, A)} (\vec{b}, \vec{a}, k)
\end{align*}
It is sufficient to show that
$$(\mathbf{(B, A)} (m+1,n+1))^{-1}\cdot \tilde r(A(0))\cdot\mathbf{(B, A)} (m,n+1) =\tilde r(A(0)).$$

If $A(0)$ is not a critical 1-cell then $\tilde r(A(0))=1$ and $p(\vec{a}) \leq k $. Since $\vec a-(\vec a)_k=(\vec a-1)-(\vec a-1)_k$, $\mathbf{(B, A)} (\vec{b}, \vec{a})=\mathbf{(B, A)} (\vec{b}, \vec{a}-1)$ and we are done.

Assume that $A(0)$ is a critical 1-cell. If $|\vec b|=0$,
$$\mathbf{(B, A)} (\vec{b}, \vec{a}-1)=\mathbf{(B, A)} (\vec{b}, \vec{a})=1$$  and we are done. If $|\vec b|=n+1$,
\begin{align*}
(\mathbf{(B, A)}& (m+1,n+1))^{-1}\cdot\tilde r(A(0))\cdot\mathbf{(B, A)} (m,n+1)\\
 =& (\mathbf{(B, A)} (m+1,n))^{-1}\cdot\mathbf{A} (m+1,n)\cdot\tilde r(B(0))^{-1}\cdot (\mathbf{A} (m+1,n))^{-1}\\
 &\cdot\tilde r(A(0))\cdot\mathbf{A} (m,n)\cdot\tilde r(B(0))\cdot (\mathbf{A} (m,n))^{-1}\cdot \mathbf{(B, A)} (m,n)
\end{align*}
where $B(0)=B_{p(\vec b)}((\vec b-1)+\vec\delta_1)$.

If $B(0)$ is not a critical 1-cell then $\tilde r(B(0))=1$ and by induction we are done.
If $B(0)$ is a critical 1-cell then $c'=A_{p(\vec a)}((\vec a-1)+\vec\delta_{k})\cup B_{p(\vec b)}((\vec b-1)+\vec\delta_1)$ is a critical 2-cell and
$c'$ satisfies the condition in Lemma~\ref{rel1}(3) and so
$$\tilde r(\partial c')=[B(0),(\mathbf{A}(m+1,n))^{-1}\cdot A(0)\cdot\mathbf{A}(m,n)].$$
Thus
\begin{align*}
(\mathbf{(B, A)}& (m+1,n+1))^{-1}\cdot\tilde r(A(0))\cdot\mathbf{(B, A)} (m,n+1)\\
 =& (\mathbf{(B, A)} (m+1,n))^{-1}\cdot\tilde r(A(0))\cdot \mathbf{(B, A)} (m,n)
\end{align*}
The proof is completed by induction.
\end{proof}

\subsection{Graphs not containing 4-nuclei}\label{ss33:P2}

Let $\Gamma$ be a graph containing none of $N_i$s. We can first assume that $\Gamma$ is a cactus graph with an outerplanar embedding. Since $\Gamma$ does not contain $N_4$, the maximal tree $T$ should be linear, in other words, there is a path in $T$ containing all the vertices with degree $\geq 3$. Also since $\Gamma$ does not contain $N_3$, every cycle contains at most two vertices of $\mathrm{deg} \geq 3$. Finally, since $\Gamma$ does not contain $N_2$, a cycle containing two vertices of degree $\geq 3$ cannot intersect other cycles. Therefore $\Gamma$ is a linear concatenation of two kinds of \emph{building blocks} in Figure~\ref{fig:bb}: a {\em star-bouquet}, and a {\em candy} which is a cycle with two star vertices.

\begin{figure}[ht]
\centering
\includegraphics[height=1.5cm]{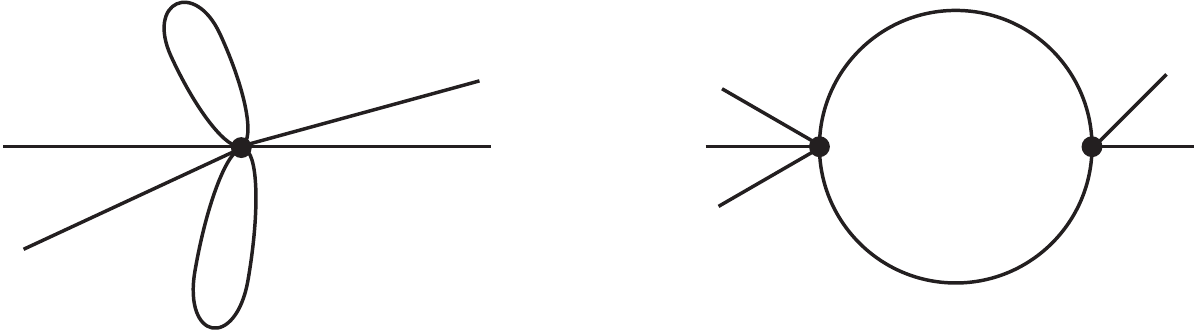}
\caption{Star-bouquet and candy}
\label{fig:bb}
\end{figure}

We take an outer-planar embedding of $\Gamma$  by placing a path containing all vertices of degree $\geq 3$ on the $x$-axis and placing all other vertices in the lower half plane except a vertex at the right end. Then we choose the vertex at the right end as the base vertex as in Figure~\ref{fig:emb}. The procedure in \S\ref{ss31:DM} gives a maximal tree $T$ of $\Gamma$ and a numbering on vertices.

\begin{figure}[ht]
\centering
\includegraphics[height=2cm]{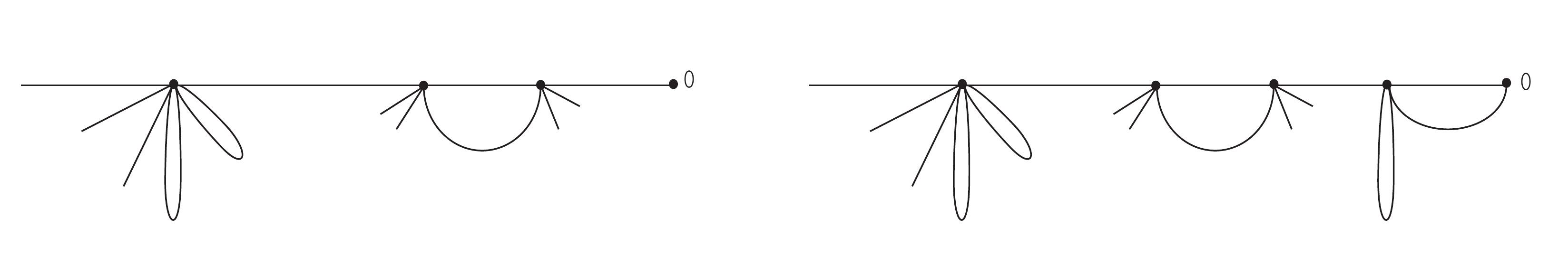}
\caption{Figure of the embeddings}
\label{fig:emb}
\end{figure}

Then we have the following property:
\begin{enumerate}
\item[(T5)] For any distinct building blocks $\mathcal B_1$ and $\mathcal B_2$,  let $m_1$ and $m_2$ ($M_1$ and $M_2$) be the smallest (largest, respectively) vertices of $\mathcal B_1$ and $\mathcal B_2$. Then either $M_1 \leq m_2$ or $M_2 \leq m_1$.
\end{enumerate}
Furthermore, for any vertices $A$, $B$ of of degree $\geq 3$ such that $A < B$, we have  $A\wedge B = A$ and $g(A, B) = \mu(A)$.

%Also for a type II building block $\mathcal B$, The two vertices $v_1 < v_2$ of degree %$\geq 3$ will be called \emph{head} and \emph{tail} of $\mathcal{B}$, denoted by $H$ %and $T$, respectively.

Using Property (T5) and Lemma~\ref{rel1}, it is easy to see the following lemma.

\begin{lem}\label{rel2}
Let $\Gamma$ be a graph contain none of $N_i$'s. Let $c$ be a critical 2-cell in $UD_n\Gamma$ of the form $A_k(\vec a)\cup B_\ell(\vec b)$ with $A<B$. Then
    \begin{enumerate}
        \item The case $A \wedge B < A$ does not occur.
        \item If $A \wedge B = A < B \wedge \iota(A_k) < B$, then $k<0$ and both $A$ and $B \wedge \iota(A_k)$ lie on a candy, and
            $$\tilde r(\partial c)=[B_\ell(\vec b),\gamma \cdot A_k(\vec a+(|\vec b|+1)\vec\delta_{|k|})]$$
            where $\gamma=\mathbf{C}((|\vec b|+1)\vec\delta_{g(C,B)},1)$ for $C=B \wedge \iota(A_k)$.
        \item If $A \wedge B = A = B \wedge \iota(A_k)$,
            $$\tilde r(\partial c)=[B_\ell(\vec b),A_k(\vec a+(|\vec b|+1)\vec\delta_{g(A,B)})]$$
        \item If $A \wedge B = A < B \wedge \iota(A_k) = B$, then $k<0<\ell$ and both $A$ and $B$ lies on a candy, and
            %$\tilde r(\partial c)$ presents the following relation
            \begin{align*}
            \tilde r(\partial c)=& B_\ell(\vec b+\vec\delta_1)\cdot \beta_2\cdot A_k(\vec a+(|\vec b|+1)\vec\delta_{|k|})\cdot B_\ell(\vec b)^{-1}\\
            &\cdot A_k(\vec a+(|\vec b|+1)\vec\delta_{|k|})^{-1}\cdot\beta_1^{-1}
            \end{align*}
            where $\beta_1=\mathbf{B}(\vec b+\vec\delta_\ell,1,1)$ and
            $\beta_2=\mathbf{B}(\vec b,1,1)$.
    \end{enumerate}
\end{lem}

\section{Tietze transformations}\label{s:four}

Let $\mathcal C_i$ be the set of critical $i$-cells in $UD_n\Gamma$. Then
$$B_n\Gamma=\langle \mathcal C_1\:|\: \tilde r(\partial c), c\in \mathcal C_2\rangle.$$
By performing Tietze transformations, we turn relators of this presentation into desired forms to prove Theorems \ref{thm:SCRG} and \ref{thm:if}.

\subsection{Cactus graphs}\label{ss41:scrg}

If $\Gamma$ is a cactus graph, Lemma~\ref{rel1} tells us that a critical 2-cell $A_k(\vec a)\cup B_\ell(\vec b)$ with $A<B$ does not produce a commutator under rewriting of its boundary only if $k<0$, that is, $A_k$ is a deleted edge and both vertices $A$ and $B$ belong to a cycle of $\Gamma$. We further divide such critical 2-cells into two classes:
$$\mathcal{S}_0=\{A_k\cup B_\ell(\vec b)\}\quad\text{and}\quad \mathcal S_4=\{A_k(\vec a)\cup B_\ell(\vec b)\:|\:\vec a\ne \vec 0\}.$$

A critical 2-cell $A_k\cup B_\ell(\vec b-\vec\delta_1)$ in $\mathcal{S}_0$ produces a relation
$$ B_\ell(\vec b)=\mathbf{B}(\vec b+\vec\delta_{|\ell|}-\vec\delta_1,1,1)\cdot A_k(|\vec b|\vec\delta_{|k|})  \cdot B_\ell(\vec b-\vec\delta_1)\cdot  (A_k(|\vec b|\vec\delta_{|k|} ))^{-1}\cdot (\mathbf{B}(\vec b-\vec\delta_1,1,1))^{-1}.$$

If $B$ is not the smallest among vertices of degree $\geq 3$ in a cycle of $\Gamma$ and $B_\ell(\vec b-\vec\delta_1)$ is a critical 1-cell, then
there is a critical 2-cell $A_k\cup B_\ell(\vec b-\vec\delta_1)$ such that $A\wedge B=A<B\wedge\iota(A_k)=B$ and the critical 1-cell $B_\ell(\vec b)$ can be written in terms of other critical 1-cells. In this case, $B_\ell(\vec b)$ is called a {\em target}.
We denote the set of all targets by $\mathcal{T}$. For $B_\ell(\vec b)\in\mathcal{T}$, the right hand side of the above relation is denoted by $\mathbf R(B_\ell(\vec b))$.

\begin{lem}\label{lem:comm}
Let $\Gamma$ be a cactus graph and let $\mathcal S_0\subset \mathcal C_2$ be defined above.
For $c=A_k(\vec a)\cup B_\ell(\vec b)\in\mathcal S_4$, if we replace $B_\ell(\vec b+\vec\delta_1)$ by $\mathbf R(B_\ell(\vec b+\vec\delta_1))$ in $\tilde r(\partial c)$, we obtain a commutator $$[\beta,\alpha_1^{-1}\omega_1^{-1}\alpha_2\omega_2]$$
where $\beta=B_\ell(\vec b)$, $\alpha_1=A_k((|\vec b|+1)\vec\delta_{|k|})$, $\omega_1=\mathbf{A} (\vec{a}, g(A,B), |\vec b|+2)$, $\alpha_2=A_k(\vec a+(|\vec b|+1)\vec\delta_{|k|})$, and $\omega_2=\mathbf{A} (\vec{a}, g(A,B) ,|\vec b|+1))$.
\end{lem}

\begin{proof}
By Lemma~\ref{rel1}(4), $\tilde r(\partial c)$ gives the relation
$$\omega_1\beta_+\omega_1^{-1}=\beta_1\alpha_2\omega_2\beta\omega_2^{-1} \alpha_2^{-1}\beta_2^{-1}$$
where $\beta_+=B_\ell(\vec b+\vec\delta_1)$, $\beta_1=(\mathbf{B},\mathbf{A})(\vec b+\vec\delta_{|\ell|},\vec a)$, and $\beta_2=(\mathbf{B},\mathbf{A})(\vec b,\vec a)$.
Then replacing $\beta_+$ by
$$\mathbf R(\beta_+)=\gamma_1\alpha_1\beta\alpha_1^{-1}\gamma_2^{-1}$$
where $\gamma_1=\mathbf{B}(\vec b+\vec\delta_{|\ell|},1,1)$, $\gamma_2=\mathbf{B}(\vec b,1,1)$ and $\alpha_1=A_k((|\vec b|+1)\vec\delta_{|k|})$,
we have
$$\omega_1\gamma_1\alpha_1\beta\alpha_1^{-1} \gamma_2^{-1}\omega_1^{-1}=\beta_1\alpha_2\omega_2\beta\omega_2^{-1} \alpha_2^{-1}\beta_2^{-1}.$$

Since $\beta_1=\omega_1\beta_3\omega_1^{-1}\beta_4$ and $\gamma_1=\beta_3\gamma_3$ by the definition of bold notations where $\beta_3=B_{p(\vec{b} + \vec{\delta_{|\ell|}})} ((\vec{b} + \vec{\delta_{|\ell|}}-1) + \vec{\delta_1})$, $\beta_4=\mathbf{(B, A)} (\vec{b} + \vec{\delta_{|\ell|}} -1, \vec{a})$ and $\gamma_3=\mathbf{B} (\vec{b} + \vec{\delta_{|\ell|}} -1 ,1, 1)$, we can modify the relation as follow:
$$\beta\alpha_1^{-1}\gamma_2^{-1}\omega_1^{-1}\beta_2\alpha_2\omega_2 =\alpha_1^{-1}\gamma_3^{-1}\omega_1^{-1}\beta_4\alpha_2\omega_2\beta$$

By Lemma~\ref{lem:AB}, $\beta_4^{-1}\omega_1\gamma_3=\beta_2^{-1}\omega_1\gamma_2=\omega_1$ and so we have
$$\beta\alpha_1^{-1}\omega_1^{-1}\alpha_2\omega_2 =\alpha_1^{-1}\omega_1^{-1}\alpha_2\omega_2\beta.$$
\end{proof}

For any vertex $C$ and for any critical 1-cell $C_m(\vec c)$ in the word $\mathbf{C}(\vec d, 1,1)$, the edge of $C_m(\vec c-\vec\delta_1)$ cannot be a critical 1-cell and so $C_m(\vec c)$ is not a target. Furthermore, if $A_k(|\vec b|\vec\delta_{|k|})$ in the word $\mathbf R(B_\ell(\vec b))$ is a target, then $A$ belongs to another cycle with smaller vertices of degree $\ge3$ and so the first component of the vector $|\vec b|\vec\delta_{|k|}$ is zero and so $A_k((|\vec b|+1)\vec\delta_{|k|}-\vec\delta_1)$ cannot be defined. Thus $A_k(|\vec b|\vec\delta_{|k|})$ is not a target. Consequently the word $\mathbf R(B_\ell(\vec b))$ contains at most two targets $B_\ell(\vec b-\vec\delta_1)$ and $A_k(\vec a+|\vec b|\vec\delta_{|k|})$.

We will perform a Tietze transformation that eliminates generators $B_\ell(\vec b)$ in $\mathcal{T}$ and replace them by $\mathbf R(B_\ell(\vec b))$ in relators $\tilde r(\partial c)$ for $c\in \mathcal C_2-\mathcal{S}_0$. The word $\mathbf R(B_\ell(\vec b))$ is never altered if $B_\ell(\vec b)$ is replaced before $B_\ell(\vec b')$ is whenever $b_1> b'_1$. Given a vertex $\vec b$, at most $n-1$ successive replacements for $B_\ell(\vec b)\in\mathcal T$ can be made where $n$ is the braid index. We denote this recursive replacement done on all generators in $\mathcal T$ by a function $s: \langle \mathcal C_1\rangle\to \langle \mathcal C_1-\mathcal T\rangle$ of free groups.

The following lemma gives Theorem~\ref{thm:SCRG}.

\begin{lem}\label{scrg}
Let $\Gamma$ be a cactus graph and let $\mathcal{T}\subset \mathcal C_1$ and $\mathcal S_0\subset \mathcal C_2$ be the notations defined above.
The braid group $B_n\Gamma$ has a simple-commutator-related presentation
$$\langle \mathcal C_1-\mathcal{T} \:|\: s\circ\tilde r(\partial c), c\in \mathcal C_2-\mathcal{S}_0\rangle.$$
\end{lem}

\begin{proof}
The presentation must be simple-commutator-related by Lemma~\ref{rel1} and Lemma~\ref{lem:comm}.
\end{proof}

\subsection{Proof of Theorem ~\ref{thm:if}}\label{ss42:raag}

For $i=1,2,3$, let $\mathcal{S}_i$ be the set of critical 2-cells $c=A_k(\vec{a})\cup B_\ell(\vec{b})$ with $A<B$ satisfying the condition $(i)$ in \S3.2. Recall that critical 2-cells $c$ satisfying the condition $(4)$ in \S3.2 are classified into $\mathcal S_0$ and $\mathcal S_4$.

We assume that $\Gamma$ be a graph containing none of 4-nuclei and work on the braid index 4. By the property (T5) in \S3.3, $\mathcal{S}_1=\emptyset$. If $c\in \mathcal{S}_2$, $A$ is the smaller vertex of a candy and $B$ is a vertex in a block behind the candy. If $c\in \mathcal{S}_3$, $A$ can be any kind of vertices. If $c\in \mathcal S_0\cup\mathcal S_4$, $A$ and $B$ are two vertices of a candy.

We first recall the simple-commutator-related presentation of $B_4\Gamma$. Combining Lemma~\ref{rel2} and Lemma~\ref{lem:comm}, we have the following lemma.

\begin{lem}\label{rel3a}
Let $\Gamma$ be a graph containing no 4-nuclei. Let $c=A_k(\vec a)\cup B_\ell(\vec b)$ be a critical 2-cell for $A<B$.
    \begin{itemize}
        \item[(a)] If $c\in\mathcal{S}_2$, then $k<0$ and $$\tilde r(\partial c)=[B_\ell(\vec b),\gamma\cdot A_k(\vec a+(|\vec b|+1)\vec\delta_{|k|})]$$
            where $\gamma=\mathbf{C}(((|\vec b|+1)\vec\delta_{\mu(C)},1,1)$.
        \item[(b)] If $c\in\mathcal{S}_3$, then $$\tilde r(\partial c)=[B_\ell(\vec b),A_k(\vec a+(|\vec b|+1)\vec\delta_{\mu(A)})].$$
        \item[(c)] If $c\in\mathcal{S}_4$, then $k<0$ and the replacement of $B_\ell(\vec b+\vec\delta_1)$ by $\mathbf R(B_\ell(\vec b+\vec\delta_1))$ in $\tilde r(\partial c)$ gives
            $$[B_\ell(\vec b),(A_k((|\vec b|+1)\vec\delta_{|k|}))^{-1}\cdot A_k(\vec a+(|\vec b|+1)\vec\delta_{|k|})].$$
        \item[(d)] If $B_\ell(\vec b)\in\mathcal{T}$, $B_\ell$ is not a deleted edge.
    \end{itemize}
\end{lem}
\begin{proof} (a) and (b) are special cases of (2) and (3) in Lemma~\ref{rel2}, respectively and (c) is a special case of Lemma~\ref{lem:comm}. If $B_\ell$ is a deleted edge, two cycles in $\Gamma$ share the vertex $B$, that is, $B$ is the unique vertex of a star-bouquet. This violates the requirement of a vertex in a target and so we have (d).
\end{proof}

We remark that (d) implies that critical 1-cells $A_k(\vec a+(|\vec b|+1)\vec\delta_{|k|})$ in (a) and (c) are not targets. In the previous section we have remarked that every critical 1-cell in $\gamma$ in (a) and $A_k((|\vec b|+1)\vec\delta_{|k|})$ in (c) are not targets.
Also (d) says that given a vertex $B$, at most 2 successive replacements for $B_\ell(\vec b)\in\mathcal T$ can be made since the edge $B_\ell$ is order-respecting when $|\vec b|=0$.

\begin{lem} \label{RAAGback}
Let $c = A_k(\vec a) \cup B_\ell(\vec b) \in \mathcal{S}_2\cup\mathcal{S}_3$ such that $\tilde r(\partial c) = [B_\ell(\vec b), \omega]$. Suppose that $D_m(\vec d)$ is a critical 1-cell in $\mathcal C_1 - \mathcal T$ such that $D>\iota(A)\wedge B$ and $|\vec d| \leq |\vec b|$. Then there exists a decomposition $\omega_1\cdots\omega_n$ of $\omega$ such that $[D_m(\vec d),\omega_j]\in \tilde r\circ\partial(\mathcal{S}_2\cup\mathcal{S}_3 )$ for $1\le j\le n$. Consequently, $[D_m(\vec d), \omega]$ is a consequence of relators in $\tilde r\circ\partial(\mathcal{S}_2\cup\mathcal{S}_3 )$.
\end{lem}
\begin{proof}
First assume that $c\in \mathcal{S}_3$. Then $\omega = A_k(\vec a+(|\vec b|+1)\vec\delta_{\mu(A)})$. Now consider $c' = A_k(\vec a+(|\vec b|-|\vec c|)\vec\delta_{\mu(A)}) \cup D_m(\vec d)$. Then $c' \in \mathcal{S}_3$ and by Lemma ~\ref{rel3a}(b), $\tilde r(\partial c') = [D_m(\vec d), A_k(\vec a+(|\vec b|+1)\vec\delta_{\mu(A)})] = [D_m(\vec d), \omega]$.

Assume that $c \in \mathcal{S}_2$. Then by Lemma ~\ref{rel3a}(a), $\omega = \gamma \cdot A_k(\vec a+(|\vec b|+1)\vec\delta_{|k|})$ where $\gamma =\mathbf{C}(((|\vec b|+1)\vec\delta_{\mu(C)},1,1)$ with $C=\iota(A)\wedge B$. Now consider $c' = A_k(\vec a+(|\vec b|-|\vec d|)\vec\delta_{|k|}) \cup D_m(\vec d)$. Then $c' \in \mathcal{S}_2$ and by Lemma~\ref{rel3a}(a),
$$\tilde r(\partial c') = [D_m(\vec d), \gamma_c A_k(\vec a+(|\vec b|+1)\vec\delta_{|k|})]$$
where $\gamma_c = \mathbf{C}(((|\vec d|+1)\vec\delta_{\mu(C)},1,1)$. Recall that $\gamma_c=\prod_{\alpha=0}^{|\vec d|}C_{\mu(C)}((|\vec d|-\alpha)\vec\delta_{\mu(C)}+\vec\delta_1)$. So
$$\gamma=(\prod_{i=0}^{|\vec b|-|\vec d|-1}C_{\mu(C)}((|\vec b|-i)\vec\delta_{\mu(C)}+\vec\delta_1))\cdot\gamma_c$$

For $c'' = C_{\mu(C)}(j\vec\delta_{\mu(C)}+\vec\delta_1)\cup D_m(\vec d) \in \mathcal{S}_3$, $\tilde r(\partial c'') = [D_m(\vec d),C_{\mu(C)}((|\vec d|+1+j)\vec\delta_{\mu(C)}+\vec\delta_1)]$ where $j=(|\vec b|-i)-|\vec d|-1$. We are done by combining the result for $c\in \mathcal{S}_3$.
\end{proof}

Among all relators obtained from critical 2-cells in $\mathcal C_2 - \mathcal{S}_0$ given in Lemma~\ref{rel3a}, we will show that those that contains targets are consequences of those that do not contain targets.
Let $\mathcal{C}\subseteq \mathcal C_2 - \mathcal{S}_0$ be a set of all critical 2-cells that produce relators containing no targets in Lemma~\ref{rel3a}.

\begin{lem}\label{raag}
$B_4 \Gamma$ has a presentation $\langle \mathcal C_1 - \mathcal{T}\:|\: s\circ\tilde r(\partial c), c\in \mathcal{C} \rangle$.
\end{lem}
\begin{proof}
Let $c=A_k(\vec a)\cup B_\ell(\vec b)$ be a critical 2-cell in $\mathcal C_2 - \mathcal{S}_0$.
It is enough to show that if a relator in Lemma~\ref{rel3a} contains a target then the relator is the identity in $\langle \mathcal C_1 - \mathcal{T} | s\circ\tilde r(\partial c), c\in \mathcal{C} \rangle$.

Suppose $c\in \mathcal{S}_2$. By the remark made right after Lemma ~\ref{rel3a}, only $B_\ell(\vec b)$ can be a target in $\tilde r(\partial c)$. If $B_\ell(\vec b)$ is a target, $B$ is the larger vertex of a candy that lies behind $\iota(A)\wedge B$. Each term in the (successive) replacement of $B_\ell(\vec b)$ satisfies the hypothesis of Lemma ~\ref{RAAGback}.

Suppose $c\in \mathcal{S}_3$ and $A_k(\vec a+(|\vec b|+1)\vec\delta_{\mu(A)})$ in $\tilde r(\partial c)$ is a target. Since $A_k(\vec a-\vec\delta_1+(|\vec b|+1)\vec\delta_{\mu(A)})$ is a critical 1-cell and $A_k$ is not deleted, $|\vec a| = 2$. Thus $|\vec b| = 0$ and so $B_\ell(\vec b)$ in $\tilde r(\partial c)$ cannot be a target. Recall the replacement
$$\mathbf R(A_k(\vec a+\vec\delta_{\mu(A)}))=\mathbf{A} (\vec{c}+\vec{\delta}_k,1, 1) \cdot H_m(3\vec\delta_{|m|}) \cdot A_k(\vec c) \cdot {H_m(3\vec\delta_{|m|})}^{-1} \cdot (\mathbf{A}(\vec{c},1,1))^{-1}$$ where $\vec c=\vec a-\vec\delta_1+\vec\delta_{\mu(A)}$ and $H_m$ is the deleted edge of the cycle containing $A$. For the critical 2-cell $c' = H_m(2\vec\delta_{|m|})\cup B_\ell$, we have $\tilde r (\partial c') = [B_\ell, A_{\mu(A)} (\vec\delta_1)\cdot H_m(3\vec\delta_{|m|})]$ by Lemma ~\ref{rel3a}(b). There are no targets in $\tilde r (\partial c')$ and so $c' \in\mathcal{C}$. Note that $A_{\mu(A)} (\vec\delta_1)$ appears at the end of both expressions $\mathbf{A} (\vec{c} +\vec{\delta}_k,1, 1)$ and $\mathbf{A} (\vec{c},1, 1)$, and every term $A_p(\vec v)$ in the replacement except for $A_{\mu(A)} (\vec\delta_1)$ and $H_m(3\vec\delta_{|m|})$ satisfy that the last coordinate of $\vec v$ is not zero. Therefore, the critical 2-cell $c'' = A_p(\vec v-\vec\delta_{\mu(A)}) \cup B_\ell$ gives $T_1 (c'') = [B_\ell, A_p(\vec v)]$ and $c'' \in\mathcal{C}$. Now we are done.

Suppose $c\in \mathcal{S}_3$ and $B_\ell(\vec b)$ in $\tilde r(\partial c)$ is a target. Then $A_k(\vec a+(|\vec b|+1)\vec\delta_{\mu(A)})$ in $\tilde r(\partial c)$ is not a target. So we are done by Lemma ~\ref{RAAGback}.

Suppose $c\in \mathcal{S}_4$. Since $|\vec a|\ge1$ and $|\vec b|\le1$, $B_\ell(\vec b)$ is not a target and so $\tilde r(\partial c)$ contains no targets.
\end{proof}

We now explain an idea how to turn the presentation in the previous lemma into one for a right-angled Artin group via a series of Tietze transformations.
For $i=2,3,4$, let $\mathcal S'_i=\mathcal S_i\cap\mathcal C$ and
$H(\mathcal S'_i)$ (and $T(\mathcal S'_i)$, respectively) denote the set of critical 1-cells based at $A$ ($B$, respectively) in commutator relations in Lemma~\ref{rel3a} produced by $A_k(\vec a)\cup B_\ell(\vec b)\in\mathcal S'_i$ with $A<B$.
Then %$H(\mathcal S'_2)\cap T(\mathcal S'_2) contains critical 1-cell as the form $A_k(2\vec\delta_{|k|})$
$H(\mathcal S'_2)\cap H(\mathcal S'_3)=H(\mathcal S'_2)\cap T(\mathcal S'_4)=\emptyset$ by the property of $\Gamma$ containing no 4-nuclei.

Lemma~\ref{rel3a} says that we have three families of commutator relations as follows:
\begin{enumerate}
\item[(a)] $[a_i,\gamma_ja_j]$ for $a_j\in H(\mathcal S'_2)$ and $a_i\in T(\mathcal S'_2)$
\item[(b)] $[a_i, b_j]$ for $b_j\in H(\mathcal S'_3)$ and $a_i\in T(\mathcal S'_3)$
\item[(c)] $[b_i,a_ja_k]$ for $a_j, a_k\in H(\mathcal S'_4)$ and $b_i\in T(\mathcal S'_4)$
\end{enumerate}
where each $\gamma_j$ is a word over a subset $\mathcal D$ of $\mathcal C_1-\mathcal T$ such that $H(\mathcal S'_2)\cap\mathcal D=\emptyset$.

We turn relators of the above presentation into commutators of two generators by performing Tietze transformations as follows :
\begin{enumerate}
\item[(I)]
    \begin{enumerate}
    \item[(i)] Starting from a generator $a_j$ in (a) that is based at the largest vertex.
    \item[(ii)] Add a new generator $a'_j=\gamma_ja_j$ for each $a_j\in H(\mathcal S'_2)$ and delete $a_j$ by setting $a_j=\gamma^{-1}_ja'_j$ so that the relator $[a_i,\gamma_ja_j]$ in (a) turns into $[a_i,a'_j]$.
    \item[(iii)] If $a_j\in H(\mathcal S'_2)\cap T(\mathcal S'_2)$, replace the relator $[\gamma^{-1}_ja'_j,\gamma_ka_k]$ in (a) by $[a'_j,\gamma_ka_k]$. Notice that $[\gamma^{-1}_j,\gamma_ka_k]$ is a consequence of other relators by Lemma~\ref{RAAGback} and so $[\gamma^{-1}_ja'_j,\gamma_ka_k]$ is a consequence of $[a'_j,\gamma_ka_k]$ and other relators.
    \item[(iv)] Choose another generator $a_j$ in (a) that is based at the vertex smaller than or equal to the one we just replaced. Repeat the steps until all relators in (a) become commutator of generators.
     \end{enumerate}
\item[(II)] For each $a_i\in H(\mathcal S'_2)\cap T(\mathcal S'_3)$, replace the relator $[\gamma^{-1}_ia'_i,b_j]$ in (b) by $[a'_i,b_j]$. Notice that $[\gamma^{-1}_i,b_j]$ is a consequence of other relators by Lemma~\ref{RAAGback} and so $[\gamma^{-1}_ia'_i,b_j]$ is a consequence of $[a'_i,b_j]$ and other relators.
\item[(III)] For each $a_j, a_k\in H(\mathcal S'_2)\cap H(\mathcal S'_4)$, $a_j^{-1}a_k$ in a relator of (c) become either $(a'_j)^{-1} c_ka'_k$ or $(a'_j)^{-1}a'_k$ after replacement where $c_k$ is the last generator in the word $\gamma_k^{-1}$. In the former case, we replace $(a'_j)^{-1} c_ka'_k$ by a new generator $c'_k$. One can check that if the braid index is 4 then $c_k$ uniquely determine $a_j$ and $a_k$ and so $a'_j$ and $a'_k$. One can also check that if $c_k$ appears in other relators that are already commutators of generators, so do $a'_j$ and $a'_k$. Thus it is enough to replace $c_k$ simply by $c'_k$ in other relators that are already commutators of generators. In the latter case, introduce a new generator $a''_k=(a'_j)^{-1}a'_k$ and delete $a'_k$ by setting $a'_k=a'_ja''_k$. One can check that if $[a'_i,a'_k]$ is a relator in (a) then $[a'_i,a'_j]$ is also a relator in (a). So we can replace the relator $[a'_i,a'_ka''_k]$ in (a) by $[a'_i,a''_k]$. One can check that $a'_k$ does not appears in relatiors of (b).
\end{enumerate}

In order to prove Theorem~\ref{thm:if} with mathematical rigor, we show that the right-angled Artin group presentation derived via a series of these Tietze transformations is isomorphic to the presentation in Lemma~\ref{raag}.

For $i=1,2,3$, consider the subsets $\mathcal{H}_i$ of $\mathcal C_1 - \mathcal T$ of critical 1-cells $A_k(\vec a)$ such that
\begin{enumerate}
\item[(i)] $A$ is the smaller vertex in a candy;
\item[(ii)] $A_k$ is the unique deleted edge in the candy;
\item[(iii)] The $|k|$-th coordinate of $\vec a$, which is the last coordinate, is $i$.
\end{enumerate}
Also consider the class of critical 1-cells
$$\mathcal{H}_4=\{ C_{\mu(C)} (2\vec{\delta}_{\mu(
C)} + \vec\delta_1)\:|\: C \mbox{ is the larger vertex of degree $\ge3$ in a candy}\}.$$
Let $\mathcal{H}= \mathcal{H}_1 \cup \mathcal{H}_2 \cup \mathcal{H}_3 \cup \mathcal{H}_4$ and introduce a set $\overline{\mathcal{H}} = \lbrace \bar{a} | a \in \mathcal{H}\rbrace$ of new generators. Define a function $\beta:\mathcal C_1-\mathcal{T}\to (\mathcal C_1-\mathcal{T}-\mathcal{H})\cup\overline{\mathcal{H}}$ by $\beta(a)=\bar a$ for $a\in\mathcal{H}$ and $\beta(a)=a$ otherwise. Let $F$ be a function that associates to each $c = A_k(\vec a) \cup B_\ell(\vec b) \in\mathcal{C}$ with $A < B$ a commutator $F(c)$ over $(\mathcal C_1-\mathcal{T}-\mathcal{H})\cup\overline{\mathcal{H}}$ as follows: If $c\in\mathcal S'_4$ and $\vec a = \vec\delta_{|k|}$, $F(c)=[B_\ell(\vec b),\bar b]$ where $b = C_{\mu(C)} (2\vec\delta_{\mu(C)} + \vec\delta_1)\in \mathcal{H}_4$. Otherwise, $F(c)=[\beta(B_\ell(\vec b)),\beta(A_k(\vec a+(|\vec b|+1)\vec\delta_{\mu(A)}))]$.
The following lemma gives Theorem~\ref{thm:if}.

\begin{lem}
Let $\Gamma$ be a graph containing no 4-nuclei. Then $B_4\Gamma$ is isomorphic to a right-angled Artin group $G=\langle (\mathcal C_1 - \mathcal{T} - \mathcal{H}) \cup \overline{\mathcal{H}}\:|\: F(c), c\in\mathcal{C} \rangle$.
\end{lem}
\begin{proof}
We use the simple-commutator-related presentation $$B_4\Gamma=\langle \mathcal C_1 - \mathcal{T}\:|\: s\circ\tilde r(\partial c), c\in \mathcal{C} \rangle.$$
Define a homomorphism $\varphi:G\to B_4\Gamma$ by
$$\varphi(\beta(x)) = \begin{cases} \mathbf{C}(i\vec\delta_{\mu(C)},1,1)\cdot x & \mbox{if } x \in \mathcal{H}_i \mbox{ for }i=1,3 \mbox{ or } x \in \mathcal{H}_2\mbox{ with }|\vec a| = 2 \\ x_2^{-1} x & \mbox{if }x \in \mathcal{H}_2\mbox{ with }|\vec a| = 3 \\
x_2^{-1} x_3 & \mbox{if }x \in \mathcal{H}_4 \\
x & \mbox{otherwise}
\end{cases}$$
where $x=A_k(\vec a)\in \mathcal C_1-\mathcal{T}$, $C$ is the larger vertex in the candy containing the deleted edge $A_k$, and $x_i = A_k(i\vec\delta_{|k|})$ for $i=2,3$.
And define a homomorphism $\psi:B_4\Gamma\to G$ by
$$\psi(x) = \begin{cases} \mathbf{C}(\vec\delta_{\mu(C)},1,1)^{-1}\cdot\bar x & \mbox{if } x \in \mathcal{H}_1\\
(\mathbf{C}(2\vec\delta_{\mu(C)},1,1))^{-1}\cdot\bar x_2\cdot\bar y & \mbox{if } x \in \mathcal{H}_3\\
(\mathbf{C}(2\vec\delta_{\mu(C)},1,1))^{-1}\cdot\bar x & \mbox{if } x \in \mathcal{H}_2\mbox{ with }|\vec a| = 2 \\
 (\mathbf{C}(2\vec\delta_{\mu(C)},1,1))^{-1}\cdot\bar x_2\cdot\bar x & \mbox{if }x \in \mathcal{H}_2\mbox{ with }|\vec a| = 3 \\
\bar{x}_3\cdot \bar{x}^{-1} \cdot\bar{x}_2^{-1} & \mbox{if }x \in \mathcal{H}_4 \\
x & \mbox{otherwise}
\end{cases}$$
where $x=A_k(\vec a)\in \mathcal C_1-\mathcal{T}$, $C$ is the larger vertex in the candy containing the deleted edge $A_k$ and $x_i = A_k(i\vec\delta_{|k|})$ for $i=2,3$ and $y=C_{\mu(C)} (2\vec{\delta}_{\mu(C)} + \vec\delta_1)\in\mathcal{H}_4$.
It is easy to check that $\varphi\circ \psi$ and $\psi\circ \varphi$ are the identity map on sets of generators. Thus it suffices to show that $\varphi$ and $\psi$ are well-defined, that is, for all $c\in\mathcal C$, $\varphi\circ F(c)$ and $\psi\circ s\circ\tilde r(\partial c)$ are trivial in $B_4\Gamma$ and $G$, respectively.

We first look at $\varphi\circ F$.
If $c = A_k(\vec a) \cup B_\ell(\vec b) \in \mathcal S'_4$ and $\vec a=\vec\delta_{|k|}$, then  $F(c) = [B_\ell(\vec b), \bar{y}]$ for $y = B_{\mu(B)} (2\vec\delta_{\mu(B)} + \vec\delta_1)\in \mathcal H_4$. Note that $A$ and $B$ are vertices in a candy. So $\varphi\circ F(c)=[B_\ell(\vec b),x_2^{-1}\cdot x_3]=s\circ\tilde r(\partial c)$ where $x_i = A_k(i\vec\delta_{|k|})$ for $i=2,3$.

If $c = A_k(\vec a) \cup B_\ell(\vec b) \in \mathcal S'_4$ and $\vec a\ne\vec\delta_{|k|}$, then $F(c) = [B_\ell(\vec b), \bar{x}]$ for $x=A_k(\vec a+(|\vec b|+1)\vec\delta_{\mu(A)})$. Since $x\in \mathcal{H}_2$ and $|\vec a+(|\vec b|+1)\vec\delta_{\mu(A)}|=3$, $|\vec b|=1$ and so $$\varphi\circ F(c)=[B_\ell(\vec b),(A_k(2\vec\delta_{\mu(A)}))^{-1}\cdot A_k(\vec a+2\vec\delta_{\mu(A)}))]=s\circ\tilde r(\partial c).$$

If $c = A_k(\vec a) \cup B_\ell(\vec b) \in \mathcal S'_3$ and $x=A_k(\vec a+(|\vec b|+1)\vec\delta_{\mu(A)})\in\mathcal H_4$, then $F(c) = [\beta(z), \bar{x}]$ for $z=B_\ell(\vec b)$ and so $$\varphi\circ F(c)=[\varphi(\beta(z)),\varphi(\bar x)]=[\varphi(\beta(z)),(A'_m(2\vec\delta_{\mu(A')}))^{-1}\cdot A'_m(3\vec\delta_{\mu(A')}))]$$ where $A'$ is the smaller vertex in candy containing the vertex $A$ and the deleted edge $A'_m$. Notice that $\varphi(\bar x)=\alpha_1^{-1} \alpha_2 $ for $\alpha_1=\mathcal A(2\vec\delta_{\mu(A)},1,1)\cdot A'_m(2\vec\delta_{\mu(A')})$ and $\alpha_2=\mathcal A(2\vec\delta_{\mu(A)},1,1)\cdot A'_m(3\vec\delta_{\mu(A')})$. Let $c'=B_\ell(\vec b)\cup A'_m$ and $c''=B_\ell(\vec b)\cup A'_m(\vec\delta_{\mu(A')})$ be critical 2-cells in $\mathcal S'_2$. Then $\tilde r(\partial c')=[z,\alpha_1]$ and $\tilde r(\partial c'')=[z,\alpha_2]$. If $\beta(z)=z$, we are done. If $\beta(z)\ne z$, $z\in \mathcal H_1$ since $|\vec b|=1$ and so $\varphi(\beta(z))=\mathbf{D}(\vec\delta_{\mu(D)},1,1)\cdot z$ where $D$ is the larger vertex in the candy containing $B$. By Lemma~\ref{RAAGback}, $[\varphi(\beta(z)),\alpha_1]=1$ and $[\varphi(\beta(z)),\alpha_2]=1$ and so $\varphi\circ F(c)=1$.

If $c = A_k(\vec a) \cup B_\ell(\vec b) \in \mathcal S'_3$ and $x=A_k(\vec a+(|\vec b|+1)\vec\delta_{\mu(A)})\not\in\mathcal H_4$, then $x\not\in\mathcal H$. Let $z=B_\ell(\vec b)$. If $\beta(z)=z$, $\varphi\circ F(c)=[z,x]=\tilde r(\partial c)$. If $\beta(z)\ne z$, either $z\in\mathcal H_1$ or $z\in\mathcal H_2$ with $|\vec b|=2$ since $|\vec b|\le 2$. By Lemma~\ref{RAAGback}, $[\varphi\circ\beta(z),x]=1$ since $\tilde r(\partial c)=[z,x]$.

If $c = A_k(\vec a) \cup B_\ell(\vec b) \in \mathcal S'_2$ and $x=A_k(\vec a+(|\vec b|+1)\vec\delta_{|k|})\in\mathcal H_2$ with $|\vec a+(|\vec b|+1)\vec\delta_{|k|}|=3$, then $\varphi\circ F(c) = [\varphi\circ\beta(z), x_2^{-1}x]$ for $x_2=A_k(2\vec\delta_{|k|})$.  Let $\alpha_1=\mathcal C(i\vec\delta_{\mu(C)},1,1)\cdot x_2$ and $\alpha_2=\mathcal C(i\vec\delta_{\mu(C)},1,1)\cdot x$ for $i=|\vec b|+1$. Then $c'=B_\ell(\vec b)\cup A_k((2-i)\vec\delta_{|k|})\in\mathcal S'_2$ and $\tilde r(\partial c')=[z,\alpha_1]$ and $\tilde r(\partial c)=[z,\alpha_2]$. If $\beta(z)=z$, we are done. If $\beta(z)\ne z$, $z\in \mathcal H_1$ since $|\vec b|=1$. By Lemma~\ref{RAAGback}, $[\varphi\circ\beta(z),\alpha_1]=1$ and $[\varphi\circ\beta(z),\alpha_2]=1$ and so $\varphi\circ F(c)=1$.

If $c = A_k(\vec a) \cup B_\ell(\vec b) \in \mathcal S'_2$ and $x=A_k(\vec a+(|\vec b|+1)\vec\delta_{|k|})\in\mathcal{H}_1$, then $|\vec b|=0$ and so $\beta(z)=z$ for $z=B_\ell(\vec b)$. Thus $\varphi\circ F(c)=\tilde r(\partial c)$.

If $c = A_k(\vec a) \cup B_\ell(\vec b) \in \mathcal S'_2$ and $x=A_k(\vec a+(|\vec b|+1)\vec\delta_{|k|})\in\mathcal{H}_3$ or $x\in\mathcal{H}_2$ with $|\vec a+(|\vec b|+1)\vec\delta_{|k|}|=2$, then $\vec a+(|\vec b|+1)\vec\delta_{|k|}=i\vec\delta_{|k|}$ and $\varphi\circ F(c) = [\varphi(\beta(z)),\mathbf{C}(i\vec\delta_{\mu(C)},1,1)\cdot x]$ where $C=B\wedge \iota(A_k)$ and $z=B_\ell(\vec b)$. For $c'=B_\ell((i-1)\vec\delta_1)\cup A_k\in\mathcal S'_2$, we have $$\tilde r(\partial c')=[B_\ell((i-1)\vec\delta_1),\mathbf{C}(i\vec\delta_{\mu(C)},1,1)\cdot x].$$
If $\beta(z)\ne z$ then either $z\in\mathcal H_1$ or $z\in\mathcal H_2$ with $|\vec b|=2$ since $|\vec b|\le 2$. We can apply Lemma~\ref{RAAGback} since $|\vec b|\le (i-1)$ and obtain $\varphi\circ F(c)=1$.

Next we consider the well-definedness of $\psi$.
If $c = A_k(\vec a) \cup B_\ell(\vec b) \in \mathcal S'_4$ and $\vec a=\vec\delta_{|k|}$, then $\psi(s\circ\tilde r(\partial c))=[z,\bar y]=F(c)$ for $z=B_\ell(\vec b)$ and $y=B_{\mu(B)}(\vec\delta_1+2\vec\delta_{\mu(B)})$. Note that $z\not\in\mathcal H$ since $B_\ell$ is not a deleted edge and $|\vec b|=1$.

If $c = A_k(\vec a) \cup B_\ell(\vec b) \in \mathcal S'_4$ and $\vec a\ne\vec\delta_{|k|}$ then $\psi(s\circ\tilde r(\partial c))=[z,\bar x]=F(c)$ where $z=B_\ell(\vec b)$ and $x=A_k(\vec a+2\vec\delta_{|k|})$.

If $c = A_k(\vec a) \cup B_\ell(\vec b) \in \mathcal S'_3$ and $x=A_k(\vec a+(|\vec b|+1)\vec\delta_{\mu(A)})\not\in\mathcal H_4$, then $x\not\in\mathcal H$ and so $\psi(s\circ\tilde r(\partial c))=[\psi(z),x]$ where $z=B_\ell(\vec b)$. If $z\not\in\mathcal H$, $F(c)=[z,x]$. If $z\in\mathcal H$, then either $z\in\mathcal H_1$ or $z\in\mathcal H_2$ with $|\vec b|=2$ since $|\vec b|\le 2$. So $\psi(z)=\mathbf D(i\vec\delta_{\mu(D)},1,1)^{-1}\cdot \bar z$ where $D$ is the larger vertex in the candy containing the deleted edge $B_\ell$. Note that if $w$ is a critical 1-cell in $\mathbf D(i\vec\delta_{\mu(D)},1,1)$, then either $w=D_{\mu(D)}(\vec\delta_1)$ or $w=D_{\mu(D)}(\vec\delta_1+\vec\delta_{\mu(D)})$ and so $w$ is not in $\mathcal H$ and $\psi(w)=w$. Choose a critical 2-cell $c'\in\mathcal S'_3$ such that $\tilde r(\partial c')=[w,x]$. Then $F(c')=[w,x]$ and so $[\mathbf D(i\vec\delta_{\mu(D)},1,1),x]=1$ in $G$. Since $F(c)=[\bar z,x]$, $[\psi(z),x])=1$ in $G$.

If $c = A_k(\vec a) \cup B_\ell(\vec b) \in \mathcal S'_3$ and $x=A_k(\vec a+(|\vec b|+1)\vec\delta_{\mu(A)})\in\mathcal H_4$, then there is the smaller vertex $A'$ in the candy containing the vertex $A$ and the deleted edge $A'_m$ and so $\psi(s\circ\tilde r(\partial c))=[\psi(z),\bar x_3 \bar x^{-1} \bar x_2^{-1}]$ where $z=B_\ell(\vec b)$ and $x_i= A'_m(i\vec\delta_{|m|})$ for $i=2,3$. Let $c'=A'_m((1-|\vec b|)\vec\delta_{|m|})\cup z$ and $c''=A'_m((2-|\vec b|)\vec\delta_{|m|})\cup z$ be critical 2-cells in $\mathcal S'_2$. Then $F(c)=[\beta(z),\bar x]$, $F(c')=[\beta(z),\bar x_2]$ and $F(c'')=[\beta(z),\bar x_3]$. If $z\not\in\mathcal H$, $\psi(z)=\beta(z)=z$ and we are done. If $z\in\mathcal H$, $z\in\mathcal H_1$ and there is the larger vertex $D$ in the candy containing $B$. So $\psi(z)=w\bar z$ for $w=D_{\mu(D)}(\vec\delta_1)$. Since $c_1=A_k(\vec a) \cup w$, $c_2=A'_m((1-|\vec b|)\vec\delta_{|m|})\cup w$ and $c_3=A'_m((2-|\vec b|)\vec\delta_{|m|})\cup w$ are critical 2-cells in either $\mathcal S'_3$ or $\mathcal S'_2$, $F(c_1)=[w,\bar x]$ and $F(c_2)=[w,\bar
x_2]$ and $F(c_3)=[w,\bar x_3]$ are relators in $G$. Thus $\psi(s\circ\tilde r(\partial c))=1$ in $G$.

If $c = A_k(\vec a) \cup B_\ell(\vec b) \in \mathcal S'_2$ and the $|k|$-th coordinate of $\vec a$ is 0, then $\psi(s\circ\tilde r(\partial c))=[\psi(z),\bar x]$
for $x=A_k(\vec a+(|\vec b|+1)\vec\delta_{|k|})$ and $z=B_\ell(\vec b)$. If $\psi(z)=z$, $\psi(s\circ\tilde r(\partial c))=[z,\bar x]=F(c)$. If $\psi(z)\ne z$, either $z\in\mathcal H_1$ or $z\in\mathcal H_2$ with $|\vec b|=2$. So $\psi(z)=\mathbf D(i\vec\delta_{\mu(D)},1,1)^{-1}\cdot \bar z$ where $D$ is the larger vertex in the candy containing the deleted edge $B_\ell$ and $i$ is the $|\ell|$-th coordinate of $\vec b$ that is either 1 or 2. Let $w=D_{\mu(D)}(\vec d)$ be a critical 1-cell in $\mathbf D(i\vec\delta_{\mu(D)},1,1)$ and $c'=A_k(\vec a+(|\vec b|-|\vec d|)\vec\delta_{|k|})\cup w$ be a critical 2-cell in $\mathcal S'_2$. Since $F(c)=[\bar z,\bar x]$ and $F(c')=[w,\bar x]$ are relators in $G$, $[\psi(z),\bar x]=1$ in $G$.

If $c = A_k(\vec a) \cup B_\ell(\vec b) \in \mathcal S'_2$ and the $|k|$-th coordinate $j$ of $\vec a$ is positive, then $j$ is either 1 or 2 and $\psi(s\circ\tilde r(\partial c))=[\psi(z),\psi(\alpha)]$ where $x=A_k(\vec a+(|\vec b|+1)\vec\delta_{|k|})$, $C=B\wedge{\iota(A_k)}$, $\alpha=\mathbf{C}((|\vec b|+1)\vec\delta_{\mu(C)},1,1)\cdot x$, and $z=B_\ell(\vec b)$. Let $w_1=C_{\mu(C)}(\vec\delta_1+(|\vec b|+1)\vec\delta_{\mu(C)})$ and $w_2=C_{\mu(C)}(\vec\delta_1+(|\vec b|+2)\vec\delta_{\mu(C)})$. If $j=1$, $\psi(\alpha)=\psi(w_1^{-1})\cdot\bar x$. If $j=2$, $\psi(\alpha)=\psi(w_1^{-1}w_2^{-1})\cdot\bar x$. Since $c'=C_{\mu(C)}(\vec\delta_1)\cup z$ and $c''=C_{\mu(C)}(\vec\delta_1+\vec\delta_{\mu(C)})\cup z$ are critical 2-cells in $\mathcal S'_3$, $\psi(\tilde r(\partial c'))=[\psi(z),\psi(w_1)]=1$ and $\psi(\tilde r(\partial c''))=[\psi(z),\psi(w_2)]=1$ in $G$. Using the same argument as above, we have $[\psi(z),\bar x]=1$ in $G$. Thus $[\psi(z),\psi(\alpha)]=1$ in $G$.
This completes the proof.
\end{proof}

\section{Proof of Theorem ~\ref{thm:onlyif}}\label{s:five}
We showed in \S2.2 that $H^*(B_4N_k)$ for $k=2,3,4$ has a non-trivial triple Massey product by using their simple-commutator-related presentations and Lemma~\ref{lem:3}. We will use similar arguments. To apply Lemma~\ref{lem:3}, we need to show that chosen sets of generators satisfy the cup zero condition.

Assume that $\Gamma$ is a cactus graph. We now know that $H^1(B_4\Gamma)$ has a simple-commutator-related presentation in Lemma~\ref{scrg}. The simple-commutator relations are given in Lemma~\ref{rel1} and in Lemma~\ref{lem:comm}. We continue to use the notations for subsets $\mathcal S_0, \mathcal S_1,\cdots,\mathcal S_4\subset C_2$ and $\mathcal T\subset C_1$.

\begin{lem}\label{cupcond}
Assume that $c =A_k(\vec a)\cup B_\ell(\vec b)\in \mathcal C_2 -\mathcal{S}_0$ with $A<B$ and $s\circ \tilde r (\partial c) = [u, v]$ and $x=P_p(\vec p), y=Q_q(\vec q) \in \mathcal C_1 - \mathcal{T}$ with $P \leq Q$.
\begin{enumerate}
 \item If $\varepsilon_x(u)\varepsilon_y(v)- \varepsilon_y(u)\varepsilon_x(v) \neq 0$ and $c\not\in\mathcal{S}_2$, then $P = A$, $Q=B$, $\varepsilon_x(u)\varepsilon_y(v)=0$, and $\varepsilon_y(u)\varepsilon_x(v)\neq0$.
 \item If $\varepsilon_x(u)\varepsilon_y(v)- \varepsilon_y(u)\varepsilon_x(v) \neq 0$ and $c\in\mathcal{S}_2$, then $P=A \text{ or }B\wedge \iota(A_k)$, $Q=B$, $\varepsilon_x(u)\varepsilon_y(v)=0$, and $\varepsilon_y(u)\varepsilon_x(v)\neq0$.
 \item If $P = Q$, then $\varepsilon_x(u)\varepsilon_y(v)- \varepsilon_y(u)\varepsilon_x(v) = 0$ for any $\vec p$ and $\vec q$.
\end{enumerate}
\end{lem}
\begin{proof}
If $\varepsilon_x(u)\varepsilon_y(v)- \varepsilon_y(u)\varepsilon_x(v) \neq 0$, then $\varepsilon_x(u)\neq 0\neq\varepsilon_y(v)$ or $\varepsilon_y(u)\neq 0\neq\varepsilon_x(v)$. Note that for any $D_d(\vec d)\in\mathcal T$, any critical 1-cell whose exponential sum in $s(D_i(\vec d))$ is not zero is based at $D$.

Assume $c\not\in\mathcal{S}_2$ for (1). If $\varepsilon_x(u)\neq 0$, $P=B$ since $u=B_\ell(\vec b)$. And if $\varepsilon_y(v)\neq 0$, $Q=A$ since $v$ is a word of critical 1-cells based at $A$ up to conjugation. Since $P\le Q$, this is impossible. Thus we must have $\varepsilon_y(u)\neq 0$ and $\varepsilon_x(v)\neq 0$ and so $P=A$ and $Q=B$.

Assume $c\in\mathcal{S}_2$ for (2). If $\varepsilon_x(u)\neq 0$, $P=B$. And if $\varepsilon_y(v)\neq 0$, $Q=A$ or $Q=B\wedge\iota(A_k)$ since $v$ is a word of critical 1-cells based at $A$ or $B\wedge\iota(A_k)$. This is impossible since $A$ or $B\wedge\iota(A_k)$ are smaller than $B$. Thus we must have $\varepsilon_y(u)\neq 0$ and $\varepsilon_x(v)\neq 0$ and so $P=A\ \text{or}\ B\wedge\iota(A_k)$ and $Q=B$.

Finally (3) is a consequence of (1) and (2).
\end{proof}

\begin{lem}\label{cupcond4}
Let $y=Q_q(\vec q) \in \mathcal C_1 - \mathcal{T}$ and assume
\begin{itemize}
\item[(a)] There is a cycle containing $P$ and $Q$ where $P$ is the smallest vertex of degree $\geq 3$ on the cycle.
\item[(b)] There is no cycle containing $P$ and a vertex of degree $\ge3$ and smaller than $P$.
\end{itemize}
Let $P_p$ be the deleted edge of the cycle containing $P$ and $Q$.
For $$X = \lbrace P_p(\vec p)\:|\:(|\vec q|+1)\le \mbox{ the }|p|\mbox{-th coordinate of }\vec p \rbrace\quad\text{and}\quad Y = \lbrace y \rbrace,$$ the pair $(X,Y)$ satisfies the cup zero conditions.
\end{lem}

\begin{proof}
For the case $|\vec q| = 3$ so that $X$ is the empty set, the statement is vacuously true. So assume that $|\vec q| \leq 2$.

Let $[u,v]=s\circ\tilde r(\partial c)$ for $c=A_k(\vec a)\cup B_\ell(\vec b)\in \mathcal C_2-\mathcal{S}_0$ and $x\in X$. For nontrivial contributions in $\sum_{x \in X} \varepsilon_x (u) \varepsilon_y (v) - \varepsilon_x (v) \varepsilon_y (u)$,
we assume $\varepsilon_x(u)\varepsilon_y(v)- \varepsilon_y(u)\varepsilon_x(v) \neq 0$.
Considering the position of $P$ and $Q$ and Lemma~\ref{cupcond}, we must have $c\in \mathcal S_4$, $P=A$ and $Q=B$.
Then $$[u,v]=[s(B_\ell(\vec b)),s(\alpha_1^{-1}\omega_1^{-1}\alpha_2\omega_2)]$$
where $\alpha_1=A_k((|\vec b|+1)\vec\delta_{|k|})$, $\omega_1=\mathbf{A} (\vec{a}, g(A,B), |\vec b|+2)$, $\alpha_2=A_k(\vec a+(|\vec b|+1)\vec\delta_{|k|})$, and $\omega_2=\mathbf{A} (\vec{a}, g(A,B) ,|\vec b|+1))$.
Since $A_k$ is a deleted edge in the cycle containing $P=A$ and $B$, $p=k$. Thus $x=\alpha_1$ or $x=\alpha_2$. Since $|\vec d| \leq |\vec b|$ for each $D_d(\vec d) \in \mathbf R(B_\ell(\vec b))$ and $\epsilon_yu\ne0$, we have $|\vec q|=|\vec d|$ for some $D_d(\vec d)$ and so $|\vec q|\le|\vec b|$. Thus both $\alpha_1$, $\alpha_2$ are in $X$ and critical 1-cells in $\omega_1$ and $\omega_2$ are not in $X$. Since $\alpha_1$ and $\alpha_2$ have opposite exponents in $v$, their contributions cancel each other and so the pair $(X,Y)$ satisfies the cup zero conditions.
\end{proof}

The following lemma implies Theorem ~\ref{thm:onlyif}.

\begin{lem}\label{lem:oif}
Let $\Gamma$ be a cactus graph. If $\Gamma$ contains $N_2$, $N_3$ or $N_4$  then $H^*(B_4\Gamma)$ has a non-trivial Massey product and so $B_4\Gamma$ is not a right-angled Artin group.
\end{lem}
\begin{proof}
Assume that $\Gamma$ contains $N_2$. Then there are vertices $A, B$ of degree $\geq 3$ such that $\Gamma$ has a cycle $O$ containing $A, B$, and another cycle $O'$ such that $O \cap O' = \lbrace B \rbrace$. We can choose such $A, B$ and the base vertex so that $A<B$ and there is no cycle containing $A$ and smaller vertices of degree $\geq 3$. Denote deleted edges of $O, O'$ by $A_k$ and $B_\ell$, respectively. Let $x_i, y, z$ be critical 1-cells given by $x_i = A_k (i\vec\delta_{|k|})$ for $i = 1, 2, 3$, $y = B_\ell$, and $z = B_{|\ell|} (\vec\delta_1)$. And let $X = \lbrace y \rbrace$, $Y = \lbrace A_k(\vec a)\:|\:1\le \mbox{ the }|k|\mbox{-th coordinate of }\vec a \rbrace $, and $Z = \lbrace x_2 \rbrace$. Then the pair $(Y,Z)$ satisfies the cup zero condition by Lemma~\ref{cupcond}(3) and $(X,Y)$ satisfies the cup zero condition by lemma~\ref{cupcond4}. Critical 2-cells $A_k(\vec\delta_{|k|}) \cup B_\ell$, $A_k(2\vec\delta_{|k|}) \cup B_\ell$, $A_k(\vec\delta_{|k|}) \cup B_\ell(\vec\delta_1)$, and $A_k(\vec\delta_{|k|}) \cup B_{|\ell|}(\vec\delta_1)$ gives relators $r_1 = [y, x_1^{-1} x_2]$, $r_2 = [y, x_1^{-1} x_3]$, $r_3 = [z x_1 y x_1 ^{-1}, x_2 ^{-1} x_3]$, and $r_4 = [z, x_2 ^{-1} x_3 ]$, respectively. By Lemma~\ref{lem:3}(1), $H^* (B_4 \Gamma)$ has a non-trivial Massey product.

Next we assume that $\Gamma$ does not contain $N_2$ but contains $N_3$. Then there is a cycle containing vertices $A, B, C$ of degree $\geq 3$ and there is no other cycle containing any of them. Assume that $0 < A< B < C$. By our choices of a maximal tree and an order, $A_k$ is a deleted edge for some $k<0$. Let $x_i$, $y$, $z$ be in $\mathcal C_1 - \mathcal{T}$ such that $x_i = A_k ( (i+1) \vec\delta_{|k|})$ for $i=0,1,2$, $y = B_{\mu(B)} (\vec\delta_1)$, and $z = C_{\mu(C)} (\vec\delta_1)$. Now let $X = Z = \lbrace x_1, y, z  \rbrace$ and $Y = \lbrace A_k(\vec a)\:|\:2\le \mbox{ the }|k|\mbox{-th coordinate of }\vec a \rbrace $. Again pairs $(X,Y)$ and $(Y,Z)$ satisfy the cup zero condition by Lemma~\ref{cupcond}(3) and Lemma~\ref{cupcond4}. Critical 2-cells $A_k(\vec\delta_{|k|})\cup B_{\mu(B)}(\vec\delta_1)$, $A_k(\vec\delta_{|k|})\cup C_{\mu(C)}(\vec\delta_1)$, and $B_{\mu(B)}(\vec\delta_1) \cup C_{\mu(C)}(\vec\delta_1)$ gives relators $r_1 = [y, x_1 ^{-1} x_2]$, $r_2 = [z, x_1^{-1}x_2]$, and $r_3 = [z, x_2 x_1 y x_1 ^{-1} x_2 ^{-1}]$, respectively. By Lemma~\ref{lem:3}(2), $H^* (B_4 \Gamma)$ has a non-trivial Massey product.

Finally we assume that $\Gamma$ contains neither $N_2$ nor $N_3$ but contains $N_4$. Then there are four vertices $A, B, C$ and $D$ of degree $\ge 3$ such that the minimal induced subgraph of $\Gamma$ containing $A, B, C$, and $D$ after ignoring vertices of degree 2 is given as one of graphs in Figure~\ref{fig:sub}. We will refer to this property as ($*$).
By choosing a suitable planar embedding of our maximal tree and the base vertex, we assume that $A<B<C<D$ and $A \wedge B = A$, $B \wedge C = B \wedge D = C \wedge D = B$, and $g(A,B)=\mu(A)$. By our choices of a maximal tree and an order on vertices, no deleted edges are incident to $A$ if there is only one cycle containing $A$. Thus $\mu(A)\ge 2$. Let $m=g(A,B)$, $n=g(B,C)$, and $\ell=g(B,D)$. Then $m>1$.
\begin{figure}[ht]
\centering
\includegraphics[height=2.4cm]{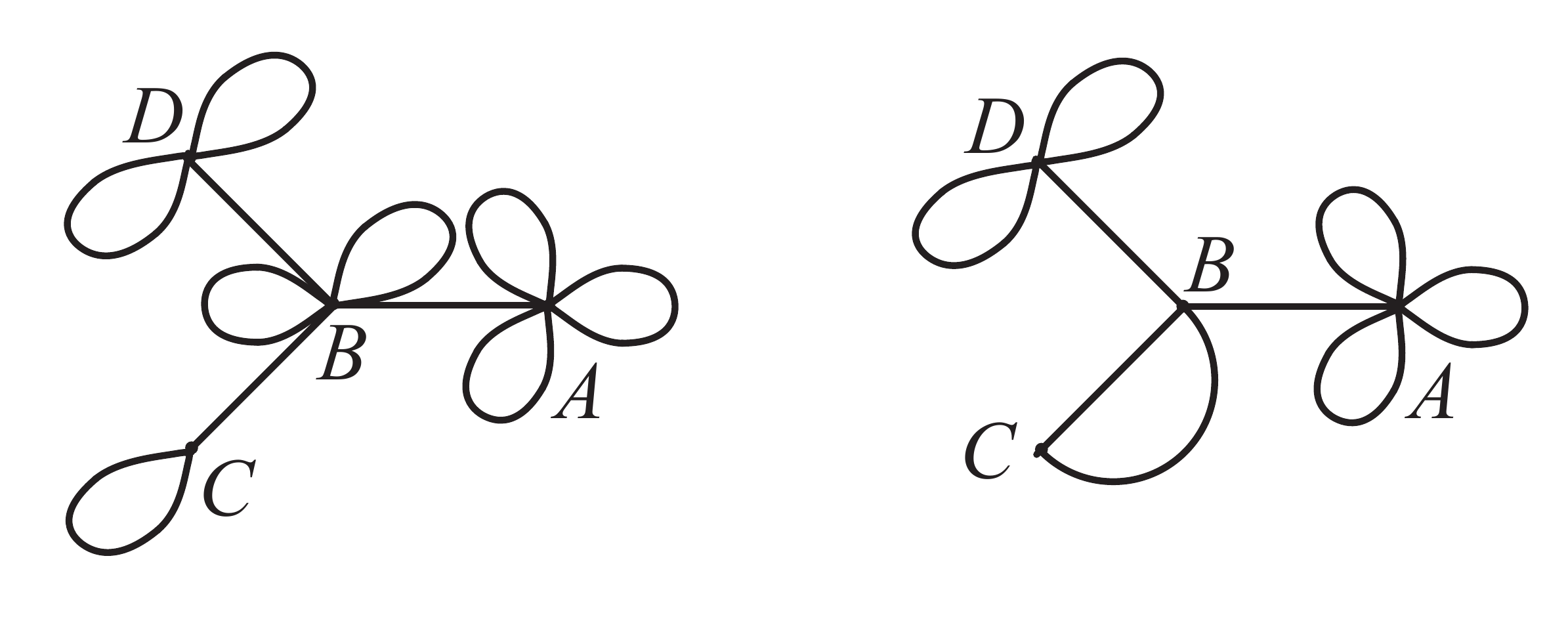}
\caption{Subgraph induced by $A,B,C,D$}
\label{fig:sub}
\end{figure}

Define critical 1-cells $x'$, $y$, $z$, and $w$ by $x' = A_m (\vec{a})$, $y = B_\ell (2\vec{\delta_n} + \vec{\delta_\ell})$, $z=C_c(\vec{c})$, and $w = D_d(\vec{d})$ with $|\vec a|=|\vec c|=|\vec d|=1$. Then by the characterization of a target, $x'$, $y$, $z$, and $w$ are not targets. Then define a critical 1-cell $x$ by $x=A_m(\vec a+(|\vec c|+1)\vec\delta_m)$. Then $A_m(\vec a+(|\vec c|+1)\vec\delta_m-\vec\delta_1)$ is not a critical 1-cell . Thus $x$ is not a target.

By Lemma~\ref{rel1}, critical 2-cells $x' \cup z$, $x' \cup w$, and $z \cup w$, gives relators $r_1 = [z, \omega_1 ^{-1} x \omega_1]$, $r_2 = [w, \omega_1^{-1} x \omega_1]$, and $r_3 = [w, y^{-1} \omega_3^{-1} z \omega_3 y]$, respectively where $\omega_i$'s are words over $\mathcal C_1 - \mathcal{T}$ that do not contain $x, y, z, w$. Note that $x,y,z,w$ are never replaced under $s$. Let $X = Z = \lbrace x, z, w \rbrace$ and $Y = \lbrace y \rbrace$. If we show that pairs $(\{x\}, \{y\})$, $(\{z\}, \{y\})$, and $(\{w\}, \{y\})$ satisfy the cup zero condition, we are done by Lemma~\ref{lem:3}(3).

Let $c = P_p(\vec p) \cup Q_q(\vec q)$ be an arbitrary critical 2-cell in $\mathcal C_2 - \mathcal{S}_0$ such that $P < Q$. And let $[u,v]=\tilde r (\partial c)$. We will derive a contradiction in each case that a pair does not satisfy the cup zero condition for $[s(u),s(v)]$.

First consider the pair $(\{x\}, \{y\})$. An analysis using Lemma~\ref{cupcond} and ($*$) leaves two possibilities: (i) $c\in\mathcal S_3$, $A=P$ and $B=Q$ or (ii) $c\in\mathcal S_2$, $A=Q\wedge\iota(P)$ and $B=Q$. Recall the expressions for $[u,v]$ from Lemma~\ref{rel1}. For (i), $y$ is in $s(u)=s(Q_q(\vec q))$ and so $|\vec q|\ge3$ since $|\vec d|\le|\vec q|$ for any $D_d(\vec d)$ in the replacement of $Q_q(\vec q)$ and $|2\vec{\delta_n} + \vec{\delta_\ell}|=3$. But $v$ contains the critical 1-cell $P_p(\vec p+(|\vec q|+1)\vec\delta_{g(P,Q)}$ that is not a target and so $|\vec q|+1\le|\vec p+(|\vec q|+1)\vec\delta_{g(P,Q)}|\le3$. This is a contradiction. For (ii), we also have $|\vec q|\ge3$ since $y$ is in
$s(u)=s(Q_q(\vec q))$ and this is contradiction by the same reason.

For the pair $(\{z\}, \{y\})$, a similar analysis leaves one possibility that $c\in\mathcal S_3\cup\mathcal S_4$, $B=P$ and $C=Q$. This is also a contradiction by a similar argument using the expression of $[u,v]$ in Lemma~\ref{lem:comm}.

For the pair $(\{w\}, \{y\})$, we have one possibility that $c\in\mathcal S_3$, $B=P$ and $D=Q$. This is also a contradiction by a similar argument.
\end{proof}

\bibliographystyle{amsplain}

\end{document}